\documentclass[smallextended,referee,envcountsect,]{svjour3}
\smartqed

\usepackage{a4}
\usepackage{amsmath}
\usepackage{amssymb}
\usepackage{color}
\usepackage{shortvrb,psfrag}
\usepackage{epsf}           
\usepackage{graphicx}       
\usepackage{subcaption}
\usepackage{epsfig}         
\usepackage{pstricks}
\usepackage{pstricks-add}
\usepackage{pst-plot}
\usepackage{dsfont}
\usepackage{tikz}
\usepackage{pgfplots}
\usepackage{pgfplotstable}
\usepackage{booktabs}

\def\a{\alpha}

\newcommand{\RP}{\ensuremath{\left[0,+\infty\right[}}
\newcommand{\RM}{\ensuremath{\left]-\infty,0\right]}}

\newcommand{\RPP}{\ensuremath{\left]0,+\infty\right[}}

\newcommand{\RX}{\ensuremath{\left]-\infty,+\infty\right]}}

\newcommand{\at}{{\widetilde \a}}

\newcommand{\reli}{\ensuremath{\text{\rm ri}\,}}

\newcommand{\NN}{{\mathbb N}}
\newcommand{\R}{{\mathbb R}}

\newcommand{\E}{{\mathbb E}}

\newcommand{\bd}{\textbf}

\newcommand{\pinf}{\ensuremath{{+\infty}}}
\newcommand{\dom}{\ensuremath{\operatorname{dom}}}
\newcommand{\inte}{\ensuremath{\operatorname{int}}}

\newcommand{\minimize}[2]{\ensuremath{\underset{\substack{{#1}}}%
{\mathrm{minimize}}\;\;#2 }}

\newcommand{\scal}[2]{{\left\langle{{#1}\mid{#2}}\right\rangle}}
 
\newcommand{\Menge}[2]{\bigg\{{#1}~\bigg|~{#2}\bigg\}} 
\newcommand{\menge}[2]{\big\{{#1}~\big |~{#2}\big\}}

\newcommand{\supp}{\mathrm{supp}\,}

\newtheorem{algorithm}{\textbf{Algorithm}}
\begin{document}

\title{Bregman proximal gradient method for linear optimization 
under entropic constraints 
}


\author{Luis M. Brice\~{n}o-Arias \and Ma\"{e}l Le~Treust}

\institute{Luis M. Brice\~{n}o Arias \at
             Universidad T\'{e}cnica Federico Santa Mar\'{i}a\\
             Departamento de Matem\'{a}tica, of. A055.\\
             Av. Vicu\~{n}a Mackenna 3939\\
             8940897, San Joaqu\'{i}n, Santiago de Chile\\
             luis.briceno@usm.cl
           \and
              Ma\"{e}l Le~Treust  \at
              Univ. Rennes, CNRS, Inria, IRISA UMR 6074\\
			 263 Av. G\'{e}n\'{e}ral Leclerc,\\
      		 35000 Rennes, France\\
              mael.le-treust@cnrs.fr
}

\date{Received: date / Accepted: date}

\maketitle

\begin{abstract}
In this paper, we present an efficient algorithm for solving a linear 
optimization problem with entropic constraints, a class of 
problems that arises in game theory and information 
theory. Our analysis distinguishes between the cases of active 
and inactive constraints, addressing each using a Bregman 
proximal gradient method with entropic Legendre functions, for 
which we establish {a convergence} rate of $O(1/n)$ in 
objective values. For a specific cost structure, our framework 
provides a theoretical justification for the well-known 
Blahut-Arimoto algorithm {and the uniqueness of the 
Lagrange multiplier associated with the entropic constraint}. In the 
active constraint setting, we 
include a bisection procedure to approximate the 
strictly positive Lagrange multiplier. The efficiency of the proposed 
method is illustrated through comparisons with standard 
optimization solvers on a representative example from game 
theory, including extensions to higher-dimensional settings.
\end{abstract}
\vspace{.3cm}


\keywords{Bregman proximal gradient algorithm \and convex 
optimization \and entropic constraints \and  {convergence rate}}

\subclass{90C25\and {65K05}}



\section{Introduction}\label{sec:Introduction}
In this paper, we propose a numerical approach for solving a 
linear optimization problem with entropic constraints.
This class of problems arises in Information Theory via 
{\emph{capacity of noisy channels}
\cite{Arimoto72,Blahut72,ChiangBoyd_TIT2004,ElGhecheChierchiaPesquet_TIT2018,%
FaustFawziSaunderson_PMLR23,Hayashi_InfoGeometry2024,Hayashi_InfoGeometry25,%
HeSaundersonFawzi_TIT24,MatzDuhamel_ITW2004,NajaAlbergeDuhamel_ICASSP2010,shannon-bell-1948}},
\emph{rate-distortion} 
problems 
{\cite{Blahut72,ChiangBoyd_TIT2004,Csiszar74,ElGhecheChierchiaPesquet_TIT2018,%
Hayashi_TIT23,Hayashi_TIT25,He2024efficient,shannon-bell-1948,Shannon59}},
\emph{coordination} problems 
\cite{CuffPermuterCover10,Cuff(ImplicitCoordination)11,LeTreust(EmpiricalCoordination)17},
{\emph{quantum information bottleneck} problems 
\cite{Hayashi_Quantum2024}, and for computing the 
\emph{reliability function} 
\cite{BenTalTeboulleCharnes_JOTA88,CsiszarKorner(Book)11};
}
{Bregman approaches for \emph{difference of convex functions} 
\cite{FaustFawziSaunderson_PMLR23}; \emph{entropic mathematical 
programming problems} 
\cite{ChandrasekaranShah_MathProg2017,Erlander81};}
Economics via \emph{rational inattention} problems 
\cite{MackowiakMatejkaWiederholt2023,MatejkaMcKay2015,Sims03},
 \emph{Bayesian persuasion and information design} 
problems \cite{DovalSkreta2024,LeTreustTomala19},  and 
\textit{repeated games} 
\cite{GossnerHernandezNeyman06,GossnerLarakiTomala09,%
GossnerTomala06,GossnerTomala07,GossnerVieille02,%
NeymanOkada99,NeymanOkada00}. In particular, 
\cite{GossnerHernandezNeyman06} {is devoted} to repeated 
 games 
with \textit{asymmetric information},
 \cite{NeymanOkada99,NeymanOkada00} study the case with 
 \textit{finite automata 
and bounded recall}, \cite{GossnerVieille02} considers the case 
with \textit{private observation}, and  
\cite{GossnerLarakiTomala09,GossnerTomala06,GossnerTomala07}
with 
\textit{imperfect monitoring}.

{The computation of the information measures introduced 
by Shannon in \cite{shannon-bell-1948,Shannon59} is analyzed in  
\cite{Arimoto72,Blahut72}. More specifically, 
one algorithm computes the capacity of a noisy 
channel stated in \cite[equation (7.1)]{cover-book-2006}, while the other one
calculates the rate-distortion function or, 
equivalently, the distortion-rate function stated in \cite[equations 
(10.11) and 
(10.160)]{cover-book-2006}. In the latter case, the convergence is 
proved in \cite{Csiszar74}. A comprehensive overview of these results 
is provided in \cite[Chap.~8]{CsiszarKorner(Book)11}. The 
Blahut-Arimoto algorithm that computes the channel capacity was 
reformulated in terms of a proximal point method in 
\cite{MatzDuhamel_ITW2004,NajaAlbergeDuhamel_ICASSP2010},
and in terms of a Bregman proximal method in 
\cite{FaustFawziSaunderson_PMLR23,Hayashi_TIT23,%
Hayashi_InfoGeometry2024,Hayashi_TIT25,%
Hayashi_InfoGeometry25,Hayashi_Quantum2024,%
HeSaundersonFawzi_TIT24,He2024efficient}.
The dual of the channel capacity and 
the rate-distortion problems are formulated in terms of geometric 
programs in \cite{ChiangBoyd_TIT2004}.
In \cite{BenTalTeboulleCharnes_JOTA88}, Lagrangian duality is 
applied to reformulate the reliability function, see 
\cite[pp.~152]{CsiszarKorner(Book)11}. 
In \cite{ElGhecheChierchiaPesquet_TIT2018}, closed-form 
expressions for the proximity operator of the Kullback-Leibler 
divergence are derived. On the other hand, efficient algorithms are 
provided from entropic regularizations of linear programs in 
\cite{Erlander81}. In addition, the problem of minimizing 
a linear function subject to linear and entropic constraints is treated in  
\cite{ChandrasekaranShah_MathProg2017}, where
the connection with entropy maximization problems and geometric 
programs is established.
}

%
%
%

{In this work, we focus on a linear optimization problem with the  
original entropic constraint} formulated in 
\cite{GossnerHernandezNeyman06}, {for a repeated game 
scenario}. {Our problem, detailed below as 
Problem~\ref{prob:main}, contains as a particular 
case the distortion-rate problem of \cite[equation 
(10.160)]{cover-book-2006}, for which the solution can be computed 
via the Blahut-Arimoto 
algorithm \cite{Arimoto72,Blahut72,Csiszar74}.} While an explicit 
solution is 
known {when} 
 $|S| = |A| = |B| = 2$ (see 
\cite[Example~2.1]{GossnerHernandezNeyman06}), the general 
case remains unsolved. 
Since the feasible set defined by the entropic constraint is 
nonempty, closed, and convex and the 
objective function is smooth (even 
linear), Problem~\ref{prob:main} can be solved by the projected 
gradient method \cite{levitin1966constrained}. However, 
this algorithm needs to compute efficiently the projection onto 
the convex set, which is not a simple task in general. 

In this work, we develop an efficient algorithm for solving 
Problem~\ref{prob:main}. Our analysis distinguishes two 
scenarios: (i) when the minimal cost is attained within the feasible 
set (inactive constraint), and (ii) when the minimal cost is attained 
outside 
this set (active constraint). In both cases, the proposed algorithms 
are derived as particular instances of a Bregman proximal gradient 
method based on an entropic Legendre function, for which we 
establish {a 
convergence} rate of $O(1/n)$ in objective values. 

In the active constraint case, the method is combined with a 
bisection algorithm to approximate the strictly positive Lagrange 
multiplier{, whose existence and uniqueness is guaranteed,} 
allowing us to solve the associated first-order optimality 
conditions. As a by-product of our analysis, we provide a 
theoretical justification for the widely used algorithm of Blahut and 
Arimoto \cite{Arimoto72,Blahut72} in the special case where the 
cost takes the form $c_{ab}^s = c_b^s \in \mathbb{R}$ for all $(a, 
b, s) \in A \times B \times S$.

{The paper is organized as follows.} In 
Sec.~\ref{sec:preliminaries}, we describe the properties of the 
non-linear constraint, and we investigate the existence of solutions 
to Problem~\ref{prob:main}. In Sec. \ref{sec:InactiveIC}, we 
characterize the cases where the non-linear constraint of 
\eqref{e:defC} is inactive and we determine the optimal solution. In 
Sec. 
\ref{sec:nonattainablecase}, we design the algorithm for the case 
where the non-linear constraint of \eqref{e:defC} is active. In 
Sec.~\ref{sec:numerical}, we 
provide numerical results illustrating the convergence and the 
optimality of the proposed algorithm.

\section{Motivation}\label{sec:motivation}
 For every nonempty finite 
set $I$, we denote by 
$\Delta(I)=\menge{x\in[0,1]^{{I}}}{\sum_{i\in 
I}x_i=1}$ the set of probability distributions over $I$, 
{$\reli(C)=\menge{x\in C}{{\rm cone}(C-x)={\rm 
span}(C-x)}$ stands 
for the relative interior of a set $C\subset \R^I$, where ${\rm 
cone}(C)$ and ${\rm span}(C)$ are the smallest cone (conical 
hull) and linear subspace (span) containing $C$, 
respectively.}
\begin{problem}\label{prob:main}
Let $A$, $B$, and $S$ be nonempty finite sets such that $|A|\ge 
2$, let $(c_{ab}^s)_{a\in A,\, b\in B}^{s\in S}\in \R^{(A\times 
B)^S}$ a vector of cost values, and let $p=(p_s)_{s\in 
S}\in{\reli(\Delta(S))}$ a probability distribution. 
The problem is to
\begin{align}
\label{e:PrimalProblem}
\minimize{q\in C}{\sum_{s\in 
S}p_s\sum_{a\in A}\sum_{b\in B}c_{ab}^sq_{ab}^s},
\end{align}
where 
\begin{equation}
\label{e:defC}
C=\Menge{q\in(\Delta(A\times 
B))^S}{\sum_{s\in 
S}p_s\sum_{a\in A}\sum_{b\in B}q_{ab}^s\ln 
\Bigg(\frac{q_{ab}^s}{\sum\limits_{a'\in A}\sum\limits_{s'\in 
S}p_{s'}q_{a'b}^{s'}}\Bigg)\leq 
0}
\end{equation}
with the convention $0\cdot\ln 0=0$.
\end{problem}

For a game theory interpretation of Problem~\ref{prob:main}, let 
us consider the {two-player} repeated coordination game with 
asymmetric information treated in 
\cite{GossnerHernandezNeyman06}, in which $S$ is the finite set 
of states of nature, $A$ and $B$ are the sets of actions of players 
1 and 2, respectively, and their common stage cost for state $s\in 
S$ and actions $a\in A$ and $b\in B$ is $c_{ab}^s\in\R$. At each 
stage $t\in\mathbb{N}\smallsetminus\{0\}$, $s_t$ is 
drawn according to a fixed i.i.d. probability distribution $(p_s)_{s\in 
S}\in\Delta(S)$ and the sequence of states up to stage 
$t$ is denoted by $s^t=(s_1,\ldots,s_t) \in S^t$. 
Both players observe past actions but, regarding the states of 
nature, player 1 is better informed than player 2. Indeed, at stage 
$t$, the first observes the infinite sequence of states $s^{\infty}$ 
whereas player 2 only observes the past states $s^{t-1}$. Hence, 
the long-run pure strategies of players 1 and 2, denoted by 
$\sigma = (\sigma_t)_{t\in\NN}$ and $\tau = (\tau_t)_{t\in\NN}$, 
are chosen in the set 
\begin{equation}
\label{e:strat}
\Sigma=\Menge{(\sigma,\tau)}{(\forall t\in\NN)\quad 
\begin{aligned}
&\sigma_t : A^{t-1}\times  B^{t-1}\times S^{\infty}  \to A\\[-2pt]
&\tau_t :  A^{t-1}\times  B^{t-1}\times S^{t-1}\to B
\end{aligned}
},
\end{equation}
{where we} use the convention that $A^0=B^0=S^0$ are 
singletons.  At each time 
$t\in\NN$, nature {draws} state $s_t$ according to $(p_s)_{s\in 
S}\in\Delta(S)$, players 1 and 
2 play actions $a_t=\sigma_t(a^{t-1},b^{t-1},s^{\infty})$ and 
$b_t=\tau_t(a^{t-1},b^{t-1},s^{t-1})$ that produce the cost $c_{a_tb_t}^{s_t}$. The players face the 
problem of choosing the pair of strategies $(\sigma,\tau)\in\Sigma$ 
to minimize 
the expected long-run cost function, i.e.,
\begin{align}
(\sigma,\tau)\mapsto \E\bigg[\frac{1}{T} \sum_{t=1}^T c_{a_tb_t}^{s_t}\bigg], 
\label{eq:TimeT}
\end{align}
as $T$ goes to infinity. In \cite[Theorems~1 \& 
2]{GossnerHernandezNeyman06}, the limit behavior when 
$T\to+\infty$ of the optimal 
expected long-run 
cost in \eqref{eq:TimeT}, turns out to be the optimal value of 
Problem~\ref{e:PrimalProblem}.

The above problem has been studied in the context of Information 
Theory in \cite{Cuff(ImplicitCoordination)11}, where the stage cost 
$c_{ab}^s\in\R$ with $(a,b,s)\in A\times B\times S$, may capture 
decoding inaccuracies  measured by a distance between $b$ and 
$s$ and/or the cost of using symbols $a$. The players 1 and 2 are 
called encoder and decoder, and aim at minimizing 
\eqref{eq:TimeT} by choosing strategies 
$(\sigma,\tau)=(\sigma_t,\tau_t)_{t\in\NN}$  in 
\begin{equation}
\label{e:stratC}
\Pi=\Menge{(\sigma,\tau)}{(\forall t\in\NN)\quad 
\begin{aligned}
&\sigma_t : S^{\infty}  \to A\\[-2pt]
&\tau_t :  A^{t-1}\to B
\end{aligned}
}.
\end{equation}
A refinement of the result in 
\cite{GossnerHernandezNeyman06} is obtained in 
\cite{Cuff(ImplicitCoordination)11}
(see also \cite[Theorem~2]{LeTreust(EmpiricalCoordination)17}), in which 
\begin{align}
\lim_{T \rightarrow + \infty} \min_{(\sigma, \tau)\in\Pi} \E\bigg[\frac{1}{T} \sum_{t=1}^T 
c_{a_tb_t}^{s_t}\bigg], 
\label{eq:AsymptoticC}
\end{align} 
is also equal to the value of Problem~\ref{prob:main}. Note that, since $\Pi\subset\Sigma$, 
this result is stronger than the one in \cite{GossnerHernandezNeyman06}.

We now interpret Problem~\ref{prob:main}. For every $(a,b,s)\in 
A\times B\times S$, 
the variable $q_{ab}^s$ represents a target conditional probability of $(a,b)$ given $s$, 
thus the objective function is the expected cost with respect to the joint probability 
distribution $\{p_sq_{ab}^s\}_{a\in A,\,b\in B}^{s\in S}\in 
\Delta(A\times B\times S)$. In 
\cite{GossnerHernandezNeyman06}, the authors prove that the probability distribution 
$\{p_sq_{ab}^s\}_{a\in A,\,b\in B}^{s\in S}$ satisfies the entropic 
constraint in \eqref{e:defC} \emph{if and 
only if} there exists a coding strategy  $(\sigma,\tau)\in\Sigma$ that induces an empirical 
distribution over the sequences $(a_t,b_t,s_t)_{t\in\NN}\in A^{\infty}\times B^{\infty}\times 
S^{\infty}$, which converges to $\{p_sq_{ab}^s\}_{a\in A,\,b\in 
B}^{s\in S}$. We denote by $\bd{a}$, $\bd{b}$, 
$\bd{s}$ the random variables drawn according to  
$\{p_sq_{ab}^s\}_{a\in A,\,b\in B}^{s\in S}$, thus the 
function appearing in \eqref{e:defC} reformulates as 
\begin{align}
\sum_{s\in S}p_s\sum_{a\in A}\sum_{b\in B}q_{ab}^s\ln 
\Bigg(\frac{q_{ab}^s}{\sum\limits_{a'\in A}\sum\limits_{s'\in 
S}p_{s'}q_{a'b}^{s'}}\Bigg)
&= I(\bd{b};\bd{s})-H(\bd{a}|\bd{b},\bd{s}),\label{eq:InformationConstraint0}
\end{align}
where 
\begin{align}
I(\bd{b};\bd{s})= \sum_{s\in S}p_s\sum_{a\in A}\sum_{b\in 
B}q_{ab}^s\ln  
\Bigg(\frac{\sum_{a'\in A}q_{a'b}^{s}}{\sum\limits_{a'\in 
A}\sum\limits_{s'\in 
S}p_{s'}q_{a'b}^{s'}}\Bigg)  
\quad \text{and}\quad 
H(\bd{a}|\bd{b},\bd{s})=  \sum_{s\in S}p_s\sum_{a\in A}\sum_{b\in 
B}q_{ab}^s\ln 
\Bigg(\frac{\sum_{a'\in 
A}q_{a'b}^{s}}{q_{ab}^s}\Bigg)\label{eq:MI_CondEntropy}
\end{align}
are interpreted as the mutual information $I(\bd{b};\bd{s})$, 
which measures the correlation between random variables 
$\bd{b}$ and 
$\bd{s}$, and the conditional entropy $H(\bd{a}|\bd{b},\bd{s})$, 
which 
measures the uncertainty of $\bd{a}$ given $\bd{b}$ and $\bd{s}$ 
(see, e.g., \cite[Chap. 2]{cover-book-2006}). 

{Now assume that, for every $(a,b,s)\in A\times B\times 
S$, $c_{ab}^s=c_b^s\in\R$. Then, for every $q\in (\Delta(A\times 
B))^S$, we have $\sum_{s\in S}p_s\sum_{a\in A}\sum_{b\in 
B}c_{ab}^sq_{ab}^s=\sum_{s\in S}p_s\sum_{b\in 
B}c_{b}^s\widehat{q}_{b}^s$,
where $\widehat{q}_b^s:=\sum_{a\in A} 
q_{ab}^s$. Moreover, \eqref{eq:InformationConstraint0}, 
\eqref{eq:MI_CondEntropy} and $H(\bd{a}|\bd{b},\bd{s}) \leq \ln|A|$ 
(see, e.g.,  \cite[Theorem~2.6.4]{cover-book-2006}) yield
\begin{align}
\sum_{a,b,s}p_s\frac{\widehat{q}_b^s}{|A|}\ln 
\left(\frac{\frac{\widehat{q}_b^s}{|A|}}
{\sum_{a',s'}p_{s'}\frac{\widehat{q}_{b}^{s'}}{|A|}}\right) 
=& \sum_{s\in S}p_s\sum_{b\in B}\widehat{q}_{b}^s\ln 
\Bigg(\frac{\widehat{q}_{b}^s}{\sum\limits_{s'\in S}p_{s'}
\widehat{q}_{b}^{s'}}\Bigg) - \ln|A|  \leq  \sum_{a,b,s}p_sq_{ab}^s\ln 
\left(\frac{q_{ab}^s}{\sum_{a',s'}p_{s'}q_{a'b}^{s'}}\right)\label{eq:},
\end{align}
which implies that, if ${q}$ is an optimal solution of 
Problem~\ref{prob:main}, then $\widehat{q}$ is an optimal solution 
of problem
\begin{align}
\label{e:ITProblem}
\min_{\widehat{q}\in\Xi}{\sum_{s\in 
S}p_s\sum_{b\in B}c_{b}^s\widehat{q}_{b}^s},
\end{align}
where $$\Xi=\Menge{\widehat{q}\in(\Delta(B))^{ 
S}}{\sum_{s\in 
S}p_s\sum_{b\in B}\widehat{q}_{b}^s\ln 
\Bigg(\frac{\widehat{q}_{b}^s}{\sum\limits_{s'\in S}p_{s'}
\widehat{q}_{b}^{s'}}\Bigg)\leq 
\ln|A|}.$$
Conversely, suppose that $\widehat{q}$ is an optimal solution to 
\eqref{e:ITProblem}. Then, it follows from \eqref{eq:} that 
$\frac{\widehat{q}}{|A|}$ is an optimal 
solution to Problem~\ref{prob:main} and both 
problems have the same optimal value.}

{In the literature, the problem in \eqref{e:ITProblem} corresponds 
to the distortion-rate function evaluated in $\ln |A|$, see 
{\cite[Equation (10.160)]{cover-book-2006}}. Since the criterion 
is linear and the set $\Xi$ is non-empty and compact, 
\cite[Theorem~2.12]{Beck2017} ensures that the set of solutions of 
\eqref{e:ITProblem} is nonempty. Note that the Lagrangian 
associated to \eqref{e:ITProblem} is very similar to that of
the rate-distortion problem, see \cite[Equation 
(8.2)]{CsiszarKorner(Book)11}. Indeed, by taking the inverse 
of the Lagrange multiplier associated to constraint $\Xi$, the latter 
coincides with the former and, hence, the first order optimality 
conditions are equivalent.
According to \cite[Lemma~8.5 \& 
Corollary~8.5]{CsiszarKorner(Book)11}, a solution 
$\widehat{Q}(\widehat{\lambda})$ to the Lagrangian of the 
rate-distortion problem
satisfies the fixed-point equation
\begin{align}
\label{e:defQhat}
(\forall b\in B)(\forall s\in S)\quad 
\widehat{Q}(\widehat{\lambda})_b^s
=&\frac{\widehat{T}(\widehat{Q}(\widehat{\lambda}))_b 
e^{-c_{b}^s/{\widehat{\lambda}}
}}{\sum\limits_{b'\in B}\widehat{T}(\widehat{Q}(\widehat{\lambda}))_{b'}
e^{- c_{b'}^s/{\widehat{\lambda}} } },\\
\widehat{T}\colon q&\mapsto \left(\sum_{s\in S} p_s 
q_b^s\right)_{b\in B},
\end{align}
where $\widehat{\lambda}>0$ is the Lagrange multiplier associated 
to the entropic constraint, when it is active. 

Blahut and Arimoto in 
\cite{Arimoto72,Blahut72} provide an algorithm for solving the 
Lagrangian dual objective function of \eqref{e:ITProblem}, whose 
convergence is proved in \cite[Theorem~1]{Csiszar74}. Given a 
starting probability 
$(t_{b,0})_{b\in B}\in\Delta(B)$, the algorithm iterates 
\begin{align}
&\text{For }n=0,1,2,...\nonumber\\
&\left\lfloor 
\begin{array}{lll}
&\text{For }b\in B \\
&\left\lfloor 
\begin{array}{ll}
&\text{For }s\in S \\
&\left\lfloor 
\begin{array}{ll}
&q_{b,n}^s = \frac{t_{b,n} e^{-c_{b}^s/\widehat{\lambda} 
}}{\sum_{b'\in B}t_{b'\!,n} 
e^{- c_{b'}^s/\widehat{\lambda} } }\\[3pt]
\end{array}
\right.\\[5pt]
&t_{b,n+1} = \sum_{s\in S} p_s q_{b,n}^s.
\end{array}
\right.\\ 
\end{array}
\right. 
\label{eq:IterationBlahutArimoto1}
\end{align}
This method, also described in \cite[Theorem~8.6]{CsiszarKorner(Book)11}, belongs to a famous class of the Blahut-Arimoto algorithms, which allow to compute the rate-distortion function, see \cite[equation (10.11)]{cover-book-2006}, as well as the capacity of noisy channels, see \cite[equation (7.1)]{cover-book-2006}. 
Note that, in order to solve \eqref{e:ITProblem}, {the previous} 
algorithm needs the Lagrange multiplier $\widehat{\lambda}$, which is 
not known a priori. In order to estimate $\widehat{\lambda}$, we 
usually implement  one-dimensional root finding methods, such as 
the bisection algorithm. This means that the 
Blahut-Arimoto algorithm in \eqref{eq:IterationBlahutArimoto1} has 
to be implemented for several values of $\lambda$.}

\section{Preliminaries and existence of solutions}\label{sec:preliminaries}
For every finite set $I$ and $q\in \R^I$, we define 
$\supp(q)=\menge{i\in I}{q_i\ne 0}$, for every nonempty closed 
convex 
set $X\subset \R^{I}$, let
\begin{equation}
\iota_{X}\colon q\mapsto 
\begin{cases}
0,\quad &\text{ if }q\in X;\\
\pinf,&\text{ otherwise,}
\end{cases}
\end{equation}
and define 
\begin{equation}
\label{e:Kullback}
D_I\colon \R^{I}\times\R^{I}\to\RX\colon (q,u)\mapsto 
\sum\limits_{i\in I}
\psi(q_i,u_i),
\end{equation}
where
$\psi\colon\R^2\to\RX$ is 
\begin{equation}
\label{e:defpsi}
\psi\colon (\xi,\mu)\mapsto 
\begin{cases}
\xi\ln(\xi/\mu),\quad &\text{if}\:\: 
\xi> 0\:\:\text{and}\:\:\mu>0;\\
0,&\text{if}\:\:\xi=0\:\:\text{and}\:\:\mu\ge 0;\\
\pinf,&\text{otherwise}.
\end{cases}
\end{equation}
Note that $D_I|_{\Delta(I)\times \Delta(I)}$ is the 
Kullback-Leibler 
divergence 
\cite[Section~2.3]{cover-book-2006}, which turns out to be a 
Bregman distance \cite[Example~9.10]{Beck2017}. More precisely,
{
\begin{equation}
\label{e:Bdist}
(\forall q\in \R^I)(\forall u\in\RPP^I)\quad 
D_I(q,u)=\varphi_I(q)-\varphi_I(u)-\scal{\nabla\varphi_I(u)}{q-u},
\end{equation}
where
\begin{equation}
\label{e:defvarphi}
\varphi_I\colon \R^I\to \RX\colon q\mapsto \sum_{i\in I}\phi(q_i)
\end{equation}
and 
\begin{equation}
\label{e:defphi}
\phi\colon \xi\mapsto
\begin{cases}
\xi\ln\xi,\quad&\text{if}\:\: \xi> 0;\\
0,&\text{if}\:\: \xi=0;\\
\pinf,&\text{otherwise}
\end{cases}
\end{equation}
is a convex lower semicontinuous proper function
\cite[Example~9.35]{BauschkeCombettes17}. 
It follows from \cite[Example~2.1]{Teboulle18} that $\varphi_I$ is a 
Legendre function, i.e., is 
proper, lower semicontinuous, strictly convex, and essentially 
smooth. The last property is satisfied by functions which are
 differentiable in 
the interior of its domain (assumed nonempty) and the norm of its 
gradient evaluated in sequences converging to the boundary 
of its domain goes to infinity (see, e.g., 
\cite[Definition~2.1]{Teboulle18} for more details). 
}

We state now some 
important 
properties of $\psi$ and $D_I$. The convexity of $D_I$ can be 
found in 
\cite[Theorem 2.7.2]{cover-book-2006}, {while a lower 
bound involving $\|\cdot\|_1$ is in 
\cite[Lemma~11.6.1]{cover-book-2006} and can be deduced from 
\cite[Proposition~5.1 \& Remark~5.1]{BeckTeb03}. 
For the sake of completeness, we present below a direct proof of 
convexity and draw on ideas from the proof of 
\cite[Proposition~5.1]{BeckTeb03} to obtain the lower bound.
}
\begin{proposition}
\label{p:KLprop}
Let $I$ be a finite set, let $\psi$ be the function defined 
in \eqref{e:defpsi}, and let $D_I$ the 
function defined in \eqref{e:Kullback}. Then the following 
assertions 
hold:
\begin{enumerate}
\item 
\label{p:KLpropi}
$\psi$ is proper, lower semicontinuous, and convex.
\item 
\label{p:KLpropii}
$D_I$ is proper, lower semicontinuous, convex, and satisfies
\begin{equation}
\label{e:fortmon}
(\forall q\in\Delta(I))(\forall u\in\Delta(I))\quad 
D_I(q,u)\ge\frac{1}{2}\|q-u\|^2_1.
\end{equation}
\end{enumerate}
\end{proposition}
\begin{proof}
\ref{p:KLpropi}: 
Noting that $\psi$ is the {lower semicontinuous convex 
envelope of 
the perspective 
associated to $\phi$ defined in \eqref{e:defphi},
the convexity and lower semicontinuity follow from 
\cite[Proposition~9.42]{BauschkeCombettes17}. }

\ref{p:KLpropii}: The convexity and lower semicontinuity of $D_I$
follow from \ref{p:KLpropi} in view of \eqref{e:Kullback} and 
\cite[Proposition~8.6]{BauschkeCombettes17}. Let $q$ and $u$ in 
$\Delta(I)$.
Note that, if $\supp q\not\subset \supp u$, there exists $i_0\in I$
such that $q_{i_0}>0$ and $u_{i_0}=0$, which implies from 
\eqref{e:Kullback} and \eqref{e:defpsi} that $D_I(q,u)=+\infty$ and, 
thus, the inequality holds trivially. Hence, let us assume that 
$I_q=\supp 
q\subset \supp u=I_u$.
Since $u\in\reli\Delta(I_u)$, there exists $\epsilon>0$
such that, for every $t\in [-\epsilon,1[$, 
$u+t(q-u)\in\reli\Delta(I_u)$.
Let $\varphi\colon\Delta(I_u)\to\R\colon z\mapsto 
\sum_{i\in I_u}\phi(z_i)$, where $\phi$ is defined in 
\eqref{e:defphi}, 
and set
$\theta\colon[-\epsilon,1]\to\R\colon t\mapsto 
\varphi(u+t(q-u))$. 
Note that 
$\theta$ is continuous in $[-\epsilon,1]$, continuously differentiable 
in 
$]-\epsilon,1[$,
and $\theta'\colon t\mapsto \scal{\nabla\varphi(u+t(q-u))}{q-u}$.
Then, for every $t\in]-\epsilon,1[$ we have $u+t(q-u)$ and $u$ are 
in 
$\reli\Delta(I_u)$ and \cite[Proposition~5.1(a)]{BeckTeb03} 
implies
\begin{align}
\theta'(t)-\theta'(0)&=\scal{\nabla\varphi(u+t(q-u))
-\nabla\varphi(u)}{q-u}\nonumber\\
&=
\frac{1}{t}\scal{\nabla\varphi(u+t(q-u))-\nabla\varphi(u)}{u+t(q-u)-u}
\nonumber\\
&\ge 
\frac{1}{t}\|t(q-u)\|_1^2\nonumber\\
&=t\|q-u\|_1^2.
\end{align}
Therefore,
\begin{equation}
(\forall s\in\left]0,1\right[)\quad 
\theta(s)-\theta(0)=\int_0^s\theta'(t)dt\ge 
s\theta'(0)+\frac{s^2}{2}\|q-u\|_1^2,
\end{equation}
and taking $s\to 1$ we obtain from $(q,u)\in\Delta(I_u)^2$ that
\begin{align}
\frac{1}{2}\|q-u\|_1^2&\le 
\varphi(q)-\varphi(u)-\scal{\nabla\varphi(u)}{q-u}\nonumber\\ 
&=\sum_{i\in I_u}\phi(q_i)-\sum_{i\in 
I_u}(u_i\ln(u_i)+(\ln(u_i)+1)(q_i-u_i))
\nonumber\\
&=\sum_{i\in I_q}q_i\ln(q_i)-\sum_{i\in I_u}q_i\ln(u_i)\nonumber\\
&=\sum_{i\in I_q}q_i\ln\left(\frac{q_i}{u_i}\right)\nonumber\\
&=D_I(q,u),
\end{align}
and the proof is complete.\qed
\end{proof}

\subsection{Existence of solutions and properties of 
Problem~\ref{prob:main}}
Define 
\begin{equation}
\label{e:defg}
g\colon\R^{(A\times B)^S}\to \RX\colon q \mapsto 
\begin{cases}
\sum\limits_{s\in S}p_s\sum\limits_{a\in A}\sum\limits_{b\in 
B}\psi\bigg(q_{ab}^s,\sum\limits_{a'\in A}\sum\limits_{s'\in 
S}p_{s'}q_{a'b}^{s'}\bigg),&\text{if}\:\:
 q\in 
[0,+\infty[^{(A\times B)^S};\\
\pinf,&\text{otherwise}.
\end{cases}
\end{equation}

The following result provides the existence of solutions to 
Problem~\ref{prob:main}.
\begin{proposition}
\label{p:exist}
In the context of Problem~\ref{prob:main}, 
let $g$ be the function defined in \eqref{e:defg} and let 
$C$ be the set defined in \eqref{e:defC}. Then, the following hold.
\begin{enumerate}
\item 
\label{p:existi}
There exists $\widetilde{q}\in{\reli((\Delta(A\times 
B))^S)}$ such that
$g(\widetilde{q})<0$.

\item
\label{p:existii+}
$g$ is convex and lower semicontinuous.
\item 
\label{p:existii}
$C$ is a nonempty convex compact set.
\item 
\label{p:existiii}
Problem~\ref{prob:main} admits at least a solution.
\end{enumerate}
\end{proposition}
\begin{proof}
\ref{p:existi}:
Define $\widetilde{q}\in(\Delta(A\times B))^S$ by
\begin{equation}
(\forall s\in S)(\forall a\in A)(\forall b\in B)\quad 
\widetilde{q}_{ab}^s=\dfrac{1}{|A||B|}{>0}.
\end{equation}
It follows from \eqref{e:defg} that $g(\widetilde{q})=-\ln|A|<0$.

\ref{p:existii+}: Note that $(u_{as})_{a,s}\mapsto 
\sum_{a',s'}p_{s'}u_{a's'}$ is linear and, since $\psi$ is convex and 
lower semicontinuous in view of 
Proposition~\ref{p:KLprop}\eqref{p:KLpropi}, $g$ is convex and 
lower semicontinuous. 

\ref{p:existii}: It follows from \eqref{e:defC}, \eqref{e:defg}, and 
\eqref{e:defpsi} that 
\begin{equation}
\label{e:C2}
C=g^{-1}(\RM)\cap(\Delta(A\times B))^S.
\end{equation}
Since \ref{p:existii+} asserts that $g$ is
convex and lower semicontinuous, $g^{-1}(\RM)$ is convex and 
closed.
Moreover, since $(\Delta(A\times B))^S$ is convex and compact,
we deduce from \eqref{e:C2} that
$C$ is convex and compact. The result follows from \ref{p:existi}.

\ref{p:existiii}: It is a direct consequence of 
\cite[Theorem~2.12]{Beck2017}.\qed
\end{proof}

Note that we can write
\begin{equation}
\label{e:defg?}
g\colon q \mapsto 
\sum\limits_{s\in S}p_sD_{A\times B}(q^s,\Lambda(q)),
\end{equation}
where $\Lambda\colon \R^{(A\times B)^S}\mapsto \R^{A\times 
B}$ is 
defined 
by
$$(\forall q\in \R^{(A\times B)^S})(\forall (a,b)\in A\times B)\quad 
(\Lambda(q))_{ab}=\sum_{a'\in 
A}\sum_{s'\in S}p_{s'}q_{a'b}^{s'}.$$
Hence, the constraint in Problem~\ref{prob:main} can be 
interpreted as an averaged Kullback-Leibler divergence constraint.
The following result 
provides the convergence of a fixed point procedure to solve 
convex 
optimization problems involving Kullback-Leibler divergences and 
it is interesting in its own right.
It {is a consequence of the convergence of} the 
Bregman proximal 
gradient method in \cite{BauMOR,Teboulle18} and we deduce a
{convergence rate} of $O(1/n)$ in values.

\begin{theorem}
\label{t:convgen}
For every $s\in S$, let $\Xi_s\subset A\times B$ {and, 
for every $(a,b)\in\Xi_s$, let $d_{ab}^s\in\R$.
Moreover, let $\eta\ge 0$,} denote by 
$\Xi=\prod_{s\in S}\Xi_s$ and by $\Delta_\Xi=\prod_{s\in 
S}\Delta(\Xi_s)$, and consider the problem
\begin{equation}
\label{e:probXi}\tag{$P_{\Xi}$}
\min_{q\in\Delta_{\Xi}}F(q,\eta):={g_{\Xi}(q)}
+\eta\sum_{s\in S}
\sum_{(a,b)\in \Xi_s}d_{ab}^sq_{ab}^s,
\end{equation}
where {$g_{\Xi}$ is the restriction of $g$ 
defined in 
\eqref{e:defg} to $\RP^\Xi$.} Then 
the following hold:
\begin{enumerate}
\item\label{t:convgeni} {The set of solutions to 
\eqref{e:probXi} is 
nonempty.
\item\label{t:convgenii}  Given $q^0\in {\rm 
ri}(\Delta_\Xi)$, 
consider the sequence generated 
by the recurrence 
\begin{equation}
\label{e:recursXi}
(\forall n\in\NN)(\forall s\in S)(\forall (a,b)\in \Xi_s)\quad 
q_{ab}^{n+1,s}=
\frac{(T_\Xi(q^n))_be^{-\eta 
d_{ab}^s/p_s}}{\sum\limits_{(a',b')\in\Xi_s}
(T_\Xi(q^n))_{b'}e^{-\eta 
d_{a'b'}^s/p_s}}, 
\end{equation}
where $T_{\Xi}\colon\Delta_\Xi\to 
\Delta(B)$ is defined by 
\begin{equation}
\label{e:TXi}
T{_\Xi}\colon 
q\mapsto \left(\sum_{a\in A}\sum_{s\in S}p_sq_{ab}^s\right)_{b\in 
B}.
\end{equation}
Then, the following hold.
\begin{enumerate}
\item \label{t:convgeniia}For every solution $\overline{q}$ to 
\eqref{e:probXi}, 
\begin{equation}
\label{e:o1k}
(\forall n\in\NN)\quad 0\le 
n(F(q^{n},\eta)
-F(\overline{q},\eta))+\sum_{k=0}^{n-1}
D_B(T_{\Xi}(q^{k+1}),T_{\Xi}(q^{k}))\le 
\sum_{s\in S}p_sD_{A\times B}(\overline{q}^s,q^{0,s}).
\end{equation}
\item\label{t:convgeniib} For every solution $\overline{q}$ to 
\eqref{e:probXi}, 
$$0\le F(q^{n},\eta)
-F(\overline{q},\eta)\le \frac{1}{n}\sum_{s\in S}p_sD_{A\times 
B}(\overline{q}^s,q^{0,s})$$ 
and 
$\sum_{k=1}^{\infty}D_B(T_{\Xi}(q^{k+1}),T_{\Xi}(q^{k}))<+\infty$.

\item \label{t:convgeniic}
The sequence $(q_n)_{n\in\NN}$ converges to some solution 
$\overline{q}$ to 
\eqref{e:probXi}. 
\end{enumerate}
}

\end{enumerate}

\end{theorem}
\begin{proof}
\ref{t:convgeni}:
Note that, since $\Delta_\Xi$ is compact, the existence 
of solutions to \eqref{e:probXi} is guaranteed by 
{\cite[Theorem~2.12]{Beck2017}} and 
Proposition~\ref{p:exist}\eqref{p:existii+}.

\ref{t:convgeniia}: Note that \eqref{e:probXi} can be written 
equivalently as 
\begin{equation}
\min_{q\in [0,+\infty[^{\Xi}}f(q)+g_{\Xi}(q),
\end{equation}
where 
{\begin{equation}
\begin{cases}
f\colon \R^{\Xi}\to\left]-\infty,+\infty\right]\colon q\mapsto 
\iota_{V_\Xi}(q)+\eta
\sum_{s\in S}\sum_{(a,b)\in \Xi_s}d_{ab}^sq_{ab}^s\\
V_{\Xi}=\menge{q\in \R^{\Xi}}{(\forall s\in 
S)\:\:\sum_{(a,b)\in\Xi_s}q_{ab}^s=1}.
\end{cases}
\end{equation}
Note that $g_{\Xi}=g\circ E$, where}
$E\colon \R^{\Xi}\to\R^{(A\times B)^S}$ stands for the 
canonical extension defined by
\begin{equation}
\label{e:defE}
(\forall q\in \R^{\Xi})(\forall (a,b,s)\in A\times B\times S)\quad 
(Eq)_{ab}^s=
\begin{cases}
q_{ab}^s,&\text{if}\:\: (a,b)\in \Xi_s;\\
0,&\text{otherwise}.
\end{cases}
\end{equation}
Note that $f$ and $g_{\Xi}$ are proper, lower semicontinuous, and 
convex.
Moreover, by defining 
\begin{equation}
\label{e:defh}
h\colon \R^{\Xi}\to\left]-\infty,+\infty\right]\colon q\mapsto 
\sum_{s\in 
S}p_s\varphi_{\Xi_s}(q^s),
\end{equation}
where {$\varphi_{\Xi_s}$
is defined in \eqref{e:defvarphi}}, it follows from 
\cite[Example~1]{BauMOR} that $h$ is a Legendre function 
satisfying $\overline{\rm dom }\,h=[0,+\infty[^{\Xi}={\rm dom}\, h= 
{\rm dom}\, g_{\Xi}$, 
and by defining $\widetilde{q}\in\R^{\Xi}$ by
$$(\forall s\in S)(\forall (a,b)\in\Xi_s)\quad 
\widetilde{q}_{ab}^s=\frac{1}{|\Xi_s|},$$ we have that
$\widetilde{q}\in\dom f\cap \inte\dom h$.
Moreover, $h-g_{\Xi}$ is convex on 
$\inte\dom h=\left]0,+\infty\right[^{\Xi}$. 
Indeed, 
for every $q\in \left]0,+\infty\right[^{\Xi}$, 
it follows from \eqref{e:defg}, \eqref{e:defpsi}, and \eqref{e:defE} 
that
\begin{align}
\label{e:hminusgconv}
(h-g_{\Xi})(q)&=
\sum_{s\in S}p_s\left(\sum_{(a,b)\in \Xi^s}q_{ab}^s\ln 
q_{ab}^s-\sum_{a\in A}\sum_{b\in 
B}\psi\left((Eq)_{ab}^s,\sum_{a'\in A}\sum_{s'\in 
S}p_{s'}(Eq)_{a'b}^{s'}\right)
\right)\nonumber\\
&=
\sum_{s\in S}p_s\sum_{(a,b)\in \Xi^s}q_{ab}^s\left(\ln 
q_{ab}^s-\ln\left(\frac{q_{ab}^s}{{\sum\limits_{a'\in 
A}\sum\limits_{s'\in 
S}}p_{s'}(Eq)_{a'b}^{s'}}\right)\right)\nonumber\\
&=
\sum_{s\in S}p_s\sum_{(a,b)\in \Xi^s}q_{ab}^s\ln\left({\sum_{a'\in 
A}\sum_{s'\in 
S}p_{s'}(Eq)_{a'b}^{s'}}\right)\nonumber\\
&=
\sum_{b\in 
B}\sum_{a\in A}\sum_{s\in S}p_s(Eq)_{ab}^s\ln\left({\sum_{a'\in 
A}\sum_{s'\in 
S}p_{s'}(Eq)_{a'b}^{s'}}\right)\nonumber\\
&=\varphi_B(T_{\Xi}(Eq)),
\end{align}
which is convex because it is the composition of the convex 
function $\varphi_B$, defined in \eqref{e:defvarphi},
with the linear operators
$T{_{\Xi}}$ and $E$
defined in \eqref{e:TXi} and \eqref{e:defE}, respectively.
On the other hand, note that the fixed point recursion in 
\eqref{e:recursXi} is equivalent to 
\begin{equation}
\label{e:mirrorfbXi}
(\forall n\in\NN)\quad q^{n+1}=\arg\min_{q\in 
\R^\Xi}f(q)+{\scal{\nabla
 g_{\Xi}(q^n)}{q}}+D_h(q,q^n),
\end{equation}
where the existence and uniqueness of the 
solution to \eqref{e:mirrorfbXi} are guaranteed by 
\cite[Lemma~2]{BauMOR},  $(q^n)_{n\in\NN}\subset 
\reli(\Delta_\Xi)=\prod_{s\in 
S}\reli(\Delta(\Xi_s))$, and,
for every $p\in\left]0,+\infty\right[^\Xi$ and $q\in[0,+\infty[^\Xi$,
\begin{equation}
\label{e:breghXi}
D_h\colon (q,p)\mapsto h(q)-h(p)-\scal{\nabla 
h(p)}{q-p}=\sum_{s\in S}p_s {D_{\Xi_s}(q^s,p^s)}
\end{equation} 
is the Bregman distance associated to $h$. 
Indeed, given $n\in\NN$, since $g_{\Xi}$ is 
differentiable in 
$]0,+\infty[^{\Xi}$ and, for every $q\in \reli(\Delta_\Xi)$, 
$\nabla 
g_{\Xi}(q)= 
(p_s\ln(q_{ab}^s/(T_{\Xi}(q))_b))_{(a,b)\in\Xi_s,\,s\in S}$, 
\eqref{e:mirrorfbXi} is 
equivalent 
to
{\begin{align}
\quad q^{n+1}&=\arg\min_{q\in V_{\Xi}}
\eta\sum_{s\in S}\sum_{(a,b)\in\Xi_s}d_{ab}^sq_{ab}^s+
\sum_{s\in S}\sum_{(a,b)\in\Xi_s}p_sq_{ab}^s
\ln\left(\frac{q_{ab}^{n,s}}{(T_{\Xi}(q^n))_b}\right)
+\sum_{s\in S}
p_sD_{\Xi_s}(q^s,q^{n,s})\nonumber\\
&=\arg\min_{q\in 
V_{\Xi}}\eta\sum_{s\in S}\sum_{(a,b)\in\Xi_s}d_{ab}^sq_{ab}^s+
\sum_{s\in S}p_s\sum_{(a,b)\in\Xi_s}\phi(q_{ab}^s)-\sum_{s\in 
S}p_s\sum_{(a,b)\in\Xi_s}q_{ab}^{s}\ln\left((T_{\Xi}(q^n))_b\right).
\end{align}}
Hence, it follows from  
{\cite[Theorem~28.2]{rockafellar1997convex}} and 
Proposition~\ref{p:KLprop}\eqref{p:KLpropi} that there exists 
$\xi\in\R^S$ such that
\begin{equation}
(\forall s\in S)(\forall (a,b)\in\Xi_s)\quad 0=\eta 
\frac{d_{ab}^s}{p_s}-\frac{\xi_s}{p_s}+
\ln\left(\frac{q_{ab}^{n+1,s}}{T{_{\Xi}}(q^n)_b}\right) {+1 },
\end{equation}
{which leads to}
\begin{equation}
\label{e:almos}
(\forall s\in S)(\forall (a,b)\in\Xi_s)\quad
q_{ab}^{n+1,s}=T{_{\Xi}}(q^n)_be^{-\eta d_{ab}^s/p_s}
e^{\xi_s/p_s{-1}}
\end{equation}
and,  { using that $q^{n+1}\in V_{\Xi}$, we conclude 
that \eqref{e:almos} reduces to \eqref{e:recursXi}.
In addition, since $h$ is differentiable in $\RPP^{\Xi}$,
it follows from \eqref{e:hminusgconv} that $\nabla 
h=\nabla g_{\Xi}+T_{\Xi}^*\circ\nabla\varphi_B\circ T_{\Xi}$. Hence,
given $n\in\NN$,
we deduce from \eqref{e:hminusgconv} and \eqref{e:Bdist} that
\begin{align}
\label{e:auxd}
g_{\Xi}(q^{n+1})-g_{\Xi}(q^{n})&=
h(q^{n+1})-h(q^n)-\varphi_B(T_{\Xi}(q^{n+1}))+\varphi_B(T_{\Xi}
(q^n))\nonumber\\
&=D_h(q^{n+1},q^n)+\scal{\nabla h(q^n)}{q^{n+1}-q^n}
-\varphi_B(T_{\Xi}(q^{n+1}))+\varphi_B(T_{\Xi}
(q^n))\nonumber\\
&=D_h(q^{n+1},q^n)+\scal{\nabla g_{\Xi}(q^n)}{q^{n+1}-q^n}
-\varphi_B(T_{\Xi}(q^{n+1}))+\varphi_B(T_{\Xi}
(q^n))\nonumber\\
&\hspace{6cm}+\scal{T_{\Xi}^*\nabla 
\varphi_B(T_{\Xi}(q^n))}{q^{n+1}-q^n}\nonumber\\
&=D_h(q^{n+1},q^n)+\scal{\nabla g_{\Xi}(q^n)}{q^{n+1}-q^n}
-\varphi_B(T_{\Xi}(q^{n+1}))+\varphi_B(T_{\Xi}
(q^n))\nonumber\\
&\hspace{5cm}+\scal{\nabla 
\varphi_B(T_{\Xi}(q^n))}{T_{\Xi}(q^{n+1})-T_{\Xi}(q^n)}\nonumber\\
&=D_h(q^{n+1},q^n)+\scal{\nabla g_{\Xi}(q^n)}{q^{n+1}-q^n}
-D_B(T_{\Xi}(q^{n+1}),T_{\Xi}(q^n)).
\end{align}
Moreover, the convexity of $g_{\Xi}$
implies 
$$(\forall q\in\RP^{\Xi})\quad g_{\Xi}(q^n)-g_{\Xi}(q)\le \scal{\nabla 
g_{\Xi}(q^n)}{q^n-q},$$
which, combined with \eqref{e:auxd}, yields
\begin{align}
\label{e:propg}
(\forall q\in\RP^{\Xi})\quad 
g_{\Xi}(q^{n+1})-g_{\Xi}(q)
&\le D_h(q^{n+1},q^n)
-D_B(T_{\Xi}(q^{n+1}),T_{\Xi}(q^n))+\scal{\nabla 
g_{\Xi}(q^n)}{q^{n+1}-q}.
\end{align}
On the other hand, the first order optimality conditions of 
\eqref{e:mirrorfbXi} and the three points identity 
\cite[Lemma~3.1]{chen1993convergence}
imply
\begin{align}
\label{e:auxf}
(\forall q\in\RP^{\Xi})\quad f(q^{n+1})-f(q)
&\le\scal{\nabla 
h(q^n)-\nabla h(q^{n+1})}{q^{n+1}-q}+\scal{\nabla 
g_{\Xi}(q^n)}{q-q^{n+1}}\nonumber\\
&=D_h(q,q^n)-D_h(q,q^{n+1})-D_h(q^{n+1},q^n)+\scal{\nabla 
g_{\Xi}(q^n)}{q-q^{n+1}}.
\end{align} 
Hence, by summing up \eqref{e:auxf} and \eqref{e:propg} we 
deduce
\begin{equation}
\label{e:trueineq}
(\forall q\in\RP^{\Xi})\quad F(q^{n+1},\eta)-F(q,\eta)\le 
D_h(q,q^n)-D_h(q,q^{n+1})-D_B(T_{\Xi}(q^{n+1}),T_{\Xi}(q^n)).
\end{equation}
Note that, by taking $q=q^n$ in \eqref{e:trueineq}, we deduce that 
$(F(q^n,\eta))_{n\in\NN}$ is nonincreasing and, by taking 
$q=\overline{q}$, where $\overline{q}$ is any solution to 
\eqref{e:probXi}, we obtain
 \begin{equation}
\label{e:inequs}
0\le F(q^{n+1},\eta)-F(\overline{q},\eta)\le 
D_h(\overline{q},q^n)-D_h(\overline{q},q^{n+1})
-D_B(T_{\Xi}(q^{n+1}),T_{\Xi}(q^n)).
\end{equation}
This} 
implies that the sequence 
$(D_h(\overline{q},q^n))_{n\in\NN}$
is nonincreasing, {nonnegative} in view of  
Proposition~\ref{p:KLprop}\eqref{p:KLpropii} and, 
therefore, { for every solution $\overline{q}$ to 
\eqref{e:probXi}, we have}
\begin{equation}
\label{e:convergentlamqXi}
\big(D_h(\overline{q},q^n)\big)_{n\in\NN}\quad\text{ 
is convergent}.
\end{equation}
{
Therefore, by summing from $0$ to $n-1$ in \eqref{e:inequs}, we 
deduce
\begin{align}
\label{e:ineq}
0\le n(F(q^n,\eta)-F(\overline{q},\eta))
&\le 
\sum_{k=0}^{n-1}\big(F(q^{k+1},\eta)-F(\overline{q},\eta)\big)
\nonumber\\
&\le D_h(\overline{q},q^0)-D_h(\overline{q},q^{n})
-\sum_{k=0}^{n-1}D_B(T_{\Xi}(q^{k+1}),T_{\Xi}(q^k)),
\end{align}
 which yields the result.

\ref{t:convgeniib}: It is a direct consequence of \ref{t:convgeniia}.

\ref{t:convgeniic}: Note that \ref{t:convgeniib} implies that
$F(q^{n},\eta)\to 
F(\overline{q},\eta)$.}
Finally, since $\Delta_{\Xi}$ is compact and recalling that 
$(q^n)_{n\in\NN}\subset \reli(\Delta_{\Xi})$,
let $q^*$ be an accumulation point of $(q^n)_{n\in\NN}$, say 
$q^{k_n}\to 
q^*$. The continuity of $F(\cdot,\eta)$ relatively to 
$\Delta_{\Xi}$ implies that
$F(q^{k_n},\eta)\to 
F(q^*,\eta)=F(\overline{q},\eta)\in\R$,
 which { implies that $q^*$ is also a solution to 
 \eqref{e:probXi}.  Furthermore, since \eqref{e:breghXi} 
can be written as
\begin{equation}
(\forall n\in\NN)\quad D_h(q^*,q^n)=
\sum_{s\in S}p_s\sum_{(a,b)\in\supp 
q^{*,s}}q^{*,s}_{ab}\ln\left(\frac{q^{*,s}_{ab}}{q^{n,s}_{ab}}\right),
\end{equation}
$q^{k_n}\to q^*$ implies $D_{h}({q}^*,q^{k_n})\to 0$, and } 
we deduce from 
\eqref{e:convergentlamqXi} that $D_{h}({q}^*,q^{n})\to 0$.
We conclude that $q^n\to q^*$ from 
Proposition~\ref{p:KLprop}\eqref{p:KLpropii}.\qed
\end{proof}
{
\begin{remark}
\begin{enumerate}
\item As stated in the proof of Theorem~\ref{t:convgen},
note that the fixed point recursion in 
\eqref{e:recursXi} is equivalent to the Bregman proximal gradient 
algorithm
\begin{equation}
(\forall n\in\NN)\quad q^{n+1}=\arg\min_{q\in 
\R^\Xi}f(q)+\scal{\nabla
 g_{\Xi}(q^n)}{q-q^n}+D_h(q,q^n),
\end{equation}
where $g_{\Xi}$ is the restriction of $g$ 
defined in 
\eqref{e:defg} to $\RP^\Xi$ and $h$ is the Legendre function 
defined in \eqref{e:defh} such that $\inte\dom h=\RPP^{\Xi}$. 

\item In the proof of Theorem~\ref{t:convgen}, we deduce 
from \eqref{e:hminusgconv} that $h-g_{\Xi}$ is convex. This is 
equivalent to the $1$-smoothness relative to $h$ (see, e.g., 
\cite{BauMOR,lu2018}).
\end{enumerate}

\end{remark}
}
\section{Attainability of the minimal cost}
\label{sec:InactiveIC}
In this section we study the Problem~\ref{prob:main}
in the case where the minimal cost is attainable, i.e. when the 
non-linear constraint is inactive. We first set our 
notation. 
In the context of Problem~\ref{prob:main}, we denote by 
\begin{align}
(\forall s\in S)\quad c_{\textsf{min}}^s &= \min_{(a,b) \in A\times B} 
c_{ab}^s\label{eq:Cmin}\\
c_{\textsf{min}} &= \sum_{s\in S}p_s 
c_{\textsf{min}}^s\label{eq:Cmin2}
\end{align}
the minimal costs. The maximal costs 
$(c_{\textsf{max}}^s)_{s\in S}$ 
and $c_{\textsf{max}}$ are 
defined analogously. Throughout this paper we assume
\begin{equation}
\label{e:H0}\tag{$H_0$}
(\forall s\in S) \quad c_{\textsf{min}}^s= 0\quad\text{and}\quad 
c_{\textsf{max}}>c_{\textsf{min}}.
\end{equation}
Note that these assumptions are not restrictive. Indeed, if there 
exists $s\in S$ such that {$c_{\textsf{min}}^s\neq 0$}, 
because of the 
linearity of the objective function, Problem~\ref{prob:main} is 
equivalent to the case when we consider, for every $(a,b,s)\in 
A\times B\times S$, the cost 
$c_{ab}^s-c_{\textsf{min}}^s$ instead of $c_{ab}^s$. Moreover, 
if $c_{\textsf{max}}=c_{\textsf{min}}$ any feasible point is a 
solution. 

We also denote 
\begin{align}
(\forall s\in S)\quad I^s &= \menge{(a,b) \in A\times B}{c_{ab}^s = 
c_{\textsf{min}}^s },\\
(\forall s\in S)(\forall b\in B)\quad I^s_b &= \menge{a\in 
A}{c_{ab}^s = 
c_{\textsf{min}}^s},\\
\Delta_0&=\prod_{s\in S}\Delta(I^s),
\label{e:defDeltamin}
\end{align}
Note that, for all $s\in S$, $I^s\neq \varnothing$, there exists 
$b\in B$ such that
$I^s_b\ne\varnothing$, and $\Delta_0\neq\varnothing$.

The following result helps us to understand the 
case where the information constraint is inactive.
\begin{proposition}
\label{p:achiev}
In the context of Problem~\ref{prob:main}, let 
$\overline{q}\in(\Delta(A\times B))^S$. Then 
\begin{equation}
\overline{q}\in\Delta_0\quad \Leftrightarrow\quad 
\min_{q\in(\Delta(A\times B))^S} 
\sum_{a,b,s}p_sc_{ab}^sq_{ab}^s=\sum_{a,b,s}p_sc_{ab}^s
\overline{q}_{ab}^s=c_{\textsf{min}}.
\end{equation}
\end{proposition}

\begin{proof}
Indeed, if $\overline{q}\in\Delta_0$
we have 
\begin{align}
{c_{\textsf{min}}\le }\min_{q\in(\Delta(A\times B))^S} 
\sum_{a,b,s}p_sc_{ab}^sq_{ab}^s 
= \sum_{s\in S}p_s c_{\textsf{min}}^s\!\!\!\!\sum_{(a,b)\in I^s}  
\overline{q}_{ab}^s = \sum_{s\in S}p_s c_{\textsf{min}}^s = 
c_{\textsf{min}}.\label{e:PrimalNoIC}
\end{align}
Conversely, if $\overline{q}\notin\Delta_0$, there exists 
$s\in S$ and $(a,b)\notin I^s$ such that $\overline{q}_{ab}^s>0$.
Since $(a,b)\notin I^s$, we have $c_{ab}^s>c_{\mathsf{min}}^s$ and, 
therefore,
it follows from \eqref{eq:Cmin} that
\begin{equation}
\sum_{a,b,s}p_sc_{ab}^s\overline{q}_{ab}^s>\sum_{s\in S}p_s 
c_{\textsf{min}}^s\sum_{a\in A}\sum_{b\in B}  
\overline{q}_{ab}^s = \sum_{s\in S}p_s c_{\textsf{min}}^s = 
c_{\textsf{min}}
\end{equation}
and the proof is complete.\qed
\end{proof}

From Proposition~\ref{p:achiev} we deduce that, if there exists 
$\widetilde{q}\in 
C\cap\Delta_0$, then $\widetilde{q}$ is a solution to 
Problem~\ref{prob:main} and the value of the problem is 
$c_{\mathsf{min}}$. This 
motivates the following definition.
\begin{definition}
\label{d:achiev}
The minimal cost $c_{\textsf{min}}$ is attainable for 
Problem~\ref{prob:main} if
$C\cap\Delta_0\neq\varnothing$ or,
equivalently, if
\begin{align}
\label{e:achev}\tag{$P_0$}
\min_{q\in \Delta_0}g(q)\le 0.
\end{align}
We denote by $Z^0$ the set 
of solutions to \eqref{e:achev}.
\end{definition}

The next result provides an algorithm for approximating a 
point in $Z^0=\arg\min g(\Delta_0)$, which turns out to be 
nonempty. This is a consequence of Theorem~\ref{t:convgen}.

\begin{proposition}
\label{prop:PF}
$Z^0\neq\varnothing$. { Moreover, given $q^0\in {\rm 
ri}(\Delta_0)$, 
consider the sequence generated 
by the recurrence 
\begin{equation}
\label{e:recurs0}
(\forall n\in\NN)(\forall s\in S)(\forall (a,b)\in I^s)\quad 
q_{ab}^{n+1,s}=
\frac{ 
(T_0(q^n))_b}{\sum\limits_{(a',b')\in 
I^s}(T_0(q^n))_{b'}}, 
\end{equation}
where $T_0\colon\Delta_0\to 
\Delta(B)$ is defined by 
\begin{equation}
\label{e:T0}
T_0\colon 
q\mapsto \left(\sum_{s\in 
S}p_s\sum_{a\in I_b^s}q_{ab}^s\right)_{b\in B}.
\end{equation}
Then, for every $n\in\NN$, $g(q^n)-g(\overline{q})=O(1/n)$ and 
$q^n\to \overline{q}$, for some $\overline{q}\in Z^0$.
}
\end{proposition}
\begin{proof} 
It is a consequence of Theorem~\ref{t:convgen} with $\eta=0$
and, for every $s\in S$, $\Xi_s={I^s}$, noting that 
$\Delta_0=\prod_{s\in S}\Delta({I^s})$.\qed
\end{proof}

Proposition~\ref{prop:PF} provides a convergent algorithm for 
approximating some $\overline{q}\in Z^0$. Then, if 
$g(\overline{q})\le 
0$, then 
$c_{\mathsf{min}}$ is attainable and the solution to 
Problem~\ref{prob:main} is $\overline{q}$.
Otherwise, if $g(\overline{q})> 0$,
$c_{\textsf{min}}$ is not attainable and further techniques 
developed 
{in the next sections} should be used in order to solve 
Problem~\ref{prob:main}.

\section{General case}\label{sec:nonattainablecase}
Let us assume that $c_{\textsf{min}}$ is not attainable and
note that Problem~\ref{prob:main} can be written equivalently 
as
\begin{align}
\label{e:redprobmain} 
&\minimize{q \in (\Delta(A\times B))^S}
{\sum_{s\in S}p_s\sum_{a\in A}\sum_{b\in 
B}c_{ab}^sq_{ab}^s}\nonumber \\
&s.t.\quad g(q)\leq 0,
\end{align}
where $g$ is defined in \eqref{e:defg}.
We define the Lagrangian $\mathcal{L}\colon\R^{(A\times 
 B)^S}\times\RP\to \RX$ by
\begin{equation}
\label{e:defL}
\mathcal{L}\colon
(q,\lambda)\mapsto \sum_{s\in S}p_s\sum_{a\in A}\sum_{b\in 
B}c_{ab}^sq_{ab}^s 
+ \lambda g(q).
\end{equation}
Problem~\ref{prob:main} is equivalent to
\begin{equation}
\label{e:minmax}
\min_{q \in (\Delta(A\times B))^S}\max_{\lambda\ge 
0}\,\mathcal{L}(q,\lambda)
\end{equation}
{and we denote by $v_P\in\R$ the value of this problem.}
\begin{proposition}
\label{p:saddle}
In the context of Problem~\ref{prob:main}, suppose that 
$c_{\mathsf{min}}$ is not attainable.
Then, the following hold:
\begin{enumerate}
\item 
\label{p:saddle0}
For every 
$\lambda\in\RPP$,
the set of 
solutions to 
\begin{equation}
\label{e:minlam2}
\min_{q\in \Delta(A\times B)^S}
\mathcal{L}(q,\lambda),
\end{equation}
denoted by $Z^{\lambda}$, is nonempty.
%
\item 
\label{p:saddleii}
$\overline{q}$ is a solution to 
Problem~\ref{prob:main} if and only if
there exists 
$\overline{\lambda}>0$ such that, $\overline{q}\in 
Z^{\overline{\lambda}}$ and
$g(\overline{q})=0$. In addition, $\overline{\lambda}>0$ is 
unique if there 
exists $s\in S$ such 
that $c_{ab}^s$ is not constant as $(a,b)\in\supp \overline{q}^s$ 
varies.
\item 
\label{p:saddlei}
For every $0<\lambda_1<\lambda_2$, and every $(q_1,q_2)\in 
Z^{\lambda_1}\times Z^{\lambda_2}$, we have
$g(q_1)\ge g(q_2)$.
\end{enumerate}
\end{proposition}
\begin{proof}
\ref{p:saddle0}: It is a consequence of 
Theorem~\ref{t:convgen}\eqref{t:convgeni} with $\Xi^s\equiv 
A\times B$, $\eta=1/\lambda$, and, for every $(a,b,s)\in A\times 
B\times S$, $d_{ab}^s=p_sc_{ab}^s$.

\ref{p:saddleii}:
Note that, for every $\lambda\ge 0$, \eqref{e:defL} and 
Proposition~\ref{p:exist}\eqref{p:existii+} imply that  
$\mathcal{L} 
(\cdot,\lambda)$ is lower semicontinuous and convex. Hence, it 
follows from 
Proposition~\ref{p:exist}\eqref{p:existi}
and 
{\cite[Theorem~28.2]{rockafellar1997convex}}  that 
$\overline{q}$ is a 
solution to 
Problem~\ref{prob:main} if and only if
there {exists $\overline{\lambda}\ge0$} such that 
\begin{equation}
\label{e:optcondmainproo}
\mathcal{L} 
(\overline{q},\overline{\lambda})=\min_{q \in 
(\Delta(A\times 
B))^{ S}}\mathcal{L} (q,\overline{\lambda}),\quad 
g(\overline{q})\le0,\quad 
\text{ and }\quad \overline{\lambda}g(\overline{q})=0.
\end{equation}
Now, if 
$\overline{\lambda}=0$,  we obtain from 
\eqref{e:optcondmainproo} 
and
\eqref{e:defL} that
\begin{equation}
\mathcal{L}(\overline{q},0)=\sum_{s\in S}p_s\sum_{a\in 
A}\sum_{b\in 
B}c_{ab}^s\overline{q}_{ab}^s =\min_{q\in 
(\Delta(A\times B))^S}
\sum_{s\in S}p_s\sum_{a\in A}\sum_{b\in 
B}c_{ab}^sq_{ab}^s =c_{\mathsf{min}},\label{e:CminL0}
\end{equation}
and we deduce from Proposition~\ref{p:achiev} and 
\eqref{e:optcondmainproo} that 
$\overline{q}\in 
C\cap \Delta_0$, which contradicts the fact that
$c_{\mathsf{min}}$ is not attainable.
Therefore, $\overline{\lambda}>0$ and it follows from 
\eqref{e:optcondmainproo} that $g(\overline{q})=0$.
Now, fix $\overline{q}$ to be an arbitrary solution to 
Problem~\ref{prob:main}. To prove the uniqueness of 
$\overline{\lambda}>0$, denote by $\Delta_{\Sigma}:=\prod_{s\in 
S}\Delta(\Sigma_s)$, where,
for every $s\in S$, 
$\Sigma_s=\supp\overline{q}^s$ and note that 
$\mathcal{L}_{\Sigma}:=\mathcal{L}(\cdot,\overline{\lambda})|_{\Delta_{\Sigma}}$
 is 
convex and
differentiable in $]0,+\infty[^{\Xi}$.
Hence, we deduce from  
\cite[Theorem~28.2]{rockafellar1997convex}
that \eqref{e:optcondmainproo} implies
\begin{equation}
\label{e:uniq1}
(\forall s\in S)(\forall (a,b)\in\Sigma_s)\quad 0=\frac{\partial 
\mathcal{L}_{\Sigma}}{\partial q_{ab}^s}(\overline{q})=p_sc_{ab}^s+
\overline{\lambda}\frac{\partial 
g}{\partial q_{ab}^s}(\overline{q})+\overline{\eta}^s,
\end{equation}
for some $(\overline{\eta}^s)_{s\in S}\in \R^S$.
Hence, our hypothesis implies that there exists $\sigma\in 
S$ such that $\frac{\partial 
g}{\partial q_{ab}^{\sigma}}(\overline{q})$ is not constant 
as $(a, b) \in\Sigma_{\sigma}$ varies, 
which implies 
that 
$(\frac{\partial 
g}{\partial q_{ab}^{\sigma}}(\overline{q}))_{(a,b)\in\Sigma_{\sigma}}$
 and 
the 
vector of ones 
$\boldsymbol{1}\in \R^{\Sigma_{\sigma}}$ are linearly independent.
Thus, if $\overline{\lambda}_1>0$ and 
$\overline{\lambda}_2>0$
are two multipliers satisfying \eqref{e:optcondmainproo} and, in 
turn, $\overline{\eta}^{\sigma}_1\in\R$ and 
$\overline{\eta}^{\sigma}_2\in\R$ are such that 
\begin{equation}
\label{e:uniq2}
(\forall i\in\{1,2\})(\forall (a,b)\in\Sigma_{\sigma})\quad 
0=p_{\sigma}c_{ab}^{\sigma}+
\overline{\lambda}_i\frac{\partial 
g}{\partial q_{ab}^{\sigma}}(\overline{q})+\overline{\eta}^{\sigma}_i,
\end{equation}
we deduce 
\begin{equation}
\label{e:uniq3}
(\forall (a,b)\in\Sigma_{\sigma})\quad 
0=
(\overline{\lambda}_1-\overline{\lambda}_2)\frac{\partial 
g}{\partial q_{ab}^{\sigma}}
(\overline{q})+(\overline{\eta}^{\sigma}_1
-\overline{\eta}^{\sigma}_2),
\end{equation}
and the linear independence implies 
$\overline{\lambda}_1=\overline{\lambda}_2$. We proved that 
given any solution $\overline{q}$ to Problem~\ref{prob:main}, the 
multiplier $\overline{\lambda}>0$ satisfying 
\eqref{e:optcondmainproo} is unique. Since
\begin{multline}
v_P=\sum_{s\in S}p_s\sum_{a\in A}\sum_{b\in 
B}c_{ab}^s\overline{q}_{ab}^s
=\mathcal{L}(\overline{q},\overline{\lambda})=
\min_{q \in 
(\Delta(A\times 
B))^S}\mathcal{L} (q,\overline{\lambda})\\
\le 
\max_{\lambda\ge 0}\min_{q \in 
(\Delta(A\times 
B))^S}\mathcal{L} (q,\lambda)\le 
\min_{q \in 
(\Delta(A\times 
B))^S}\max_{\lambda\ge 0}\mathcal{L} (q,\lambda)=v_P,
\end{multline}
we deduce that $\overline{\lambda}$ is the unique solution to the 
dual 
problem
\begin{equation}
\label{e:dual}\tag{$D$}
\max_{\lambda\ge 0}\min_{q \in 
(\Delta(A\times 
B))^S}\mathcal{L} (q,\lambda).
\end{equation}
Noting that \eqref{e:dual} is independent of the primal solution 
$\overline{q}$, the proof is complete.

\ref{p:saddlei}: 
Let $0<\lambda_1<\lambda_2$ and let
$(q_1,q_2)\in 
Z^{\lambda_1}\times Z^{\lambda_2}$. We have
\begin{equation}
\mathcal{L}(q_1,\lambda_1)=\sum_{a\in A}\sum_{b\in B}\sum_{s\in 
S}p_sc_{ab}^sq_{1,ab}^s+\lambda_1 g(q_1)\le
\sum_{a\in A}\sum_{b\in B}\sum_{s\in 
S}p_sc_{ab}^sq_{2,ab}^s+\lambda_1 g(q_2), 
\end{equation} 
which implies
\begin{equation}
\sum_{a\in A}\sum_{b\in B}\sum_{s\in 
S}p_sc_{ab}^s(q_{1,ab}^s-q_{2,ab}^s)\le 
\lambda_1(g(q_2)-g(q_1)).
\end{equation}
Analogously,
\begin{equation}
\sum_{a\in A}\sum_{b\in B}\sum_{s\in 
S}p_sc_{ab}^s(q_{2,ab}^s-q_{1,ab}^s)\le \lambda_2(g(q_1)-g(q_2))
\end{equation}
and combining both inequalities we obtain
\begin{equation}
\label{e:auxine}
\lambda_2(g(q_2)-g(q_1))\le \sum_{a\in A}\sum_{b\in B}\sum_{s\in 
S}p_sc_{ab}^s(q_{1,ab}^s-q_{2,ab}^s)\le 
\lambda_1(g(q_2)-g(q_1)).
\end{equation}
We deduce $(\lambda_2-\lambda_1)(g(q_2)-g(q_1))\le 0$, which 
implies {the result}.\qed
\end{proof}
In view of Proposition~\ref{p:saddle}\eqref{p:saddleii}, in order to 
find a solution to Problem~\ref{prob:main} we need to find 
$\overline{q}\in Z^{\lambda}$ for a specific value of $\lambda>0$ 
such that $g(\overline{q})=0$. {In this context, we} use an 
algorithm whose 
convergence is a consequence of Theorem~\ref{t:convgen}.
\begin{theorem}
\label{t:propPFqlam}
Let $\lambda>0${. Given $q^0\in\reli (\Delta(A\times B)^S)$, 
consider the sequence generated 
by the recurrence 
\begin{equation}
\label{e:defPFqlam}
(\forall n\in\NN)(\forall s\in S)(\forall (a,b)\in A\times B)\quad 
q_{ab}^{n+1,s}=
\displaystyle{\frac{(T(q^n))_be^{-c_{ab}^s/\lambda}}{
\displaystyle{\sum_{a'\in
A}}\displaystyle{\sum_{b'\in 
B}}(T(q^n))_{b'}e^{-c_{a'b'}^s/\lambda}}},
\end{equation}
where $T\colon\Delta(A\times B)^S\to 
\Delta(B)$ is defined by 
\begin{equation}
\label{e:T}
T\colon 
q\mapsto \left(\sum_{s\in 
S}p_s\sum_{a\in A}q_{ab}^s\right)_{b\in B}.
\end{equation}
Then, for every $n\in\NN$, 
$\mathcal{L}(q^n,\lambda)-\mathcal{L}(\overline{q},\lambda)=O(1/n)$ 
and 
$q^n\to \overline{q}$, for some $\overline{q}\in Z^{\lambda}$.
}

\end{theorem}
\begin{proof} 
Note that, from \eqref{e:defL} and \eqref{e:defg},
\eqref{e:minlam2} is equivalent to \eqref{e:probXi} with 
$\Xi=(A\times B)^S$, $\eta=1/\lambda$, and, for every 
$(a,b,s)\in A\times B\times S$, $d_{ab}^s=p_sc_{ab}^s$.
Since in \eqref{e:probXi} we have 
$F=\mathcal{L}(\cdot,\lambda)/\lambda$, the result is a
consequence of Theorem~\ref{t:convgen}. 
\qed
\end{proof}

\begin{remark}
Note that, in the particular case when, for every $(a,b,s)\in 
{A\times 
B\times S}$, $c_{ab}^s=c_b^s$, we recover the Blahut-Arimoto 
algorithm as a particular case of the recurrence in 
Theorem~\ref{t:propPFqlam}, for a fixed $\lambda>0$.
We also obtain a $O(1/n)$ {convergence} rate in 
Lagrangian values {and the uniqueness of the multiplier in 
view of Proposition~\ref{p:saddle}\eqref{p:saddleii}.}
\end{remark}

\section{Algorithm and numerical results}

In the context of Problem \ref{prob:main}, we use a bisection 
algorithm in order to estimate the optimal Lagrange multiplier 
$\overline{\lambda}>0$. 
We first compute an upper 
bound {in the following result}.

\begin{lemma}\label{lemma:LambdaMax}
Suppose that $c_{\mathsf{min}}$ is not attainable, set
\begin{align}
\lambda_{\mathsf{max}} = \frac{c_{\mathsf{max}} - 
c_{\mathsf{min}}}{\ln |A|}>0,
\end{align} 
and let $\overline{q}_{\mathsf{max}}\in 
Z^{\lambda_{\mathsf{max}}}$ be the limit of the sequence 
$\{q^n\}_{n\in\NN}$ generated by \eqref{e:defPFqlam} with 
$\lambda=\lambda_{\mathsf{max}}$, 
guaranteed by Theorem~\ref{t:propPFqlam}.
Then{,} for every $\overline{q}_{0}\in Z^{0}$, we have
\begin{align}
g(\overline{q}_{\mathsf{max}})<0\quad \text{and}\quad 
g(\overline{q}_{0})>0.
\end{align}
\end{lemma}
\begin{proof}
First note that assumption $(H_0)$ implies that there 
exists $(a,b,s)\in 
A\times B\times S$ such that $c_{ab}^s>c_{\mathsf{min}}$ and, 
hence, $\lambda_{\mathsf{max}} > 0$.
It follows from Proposition~\ref{prop:PF} and 
{Proposition~\ref{p:saddle}\eqref{p:saddle0}} that 
$Z^0\neq \varnothing$ and $Z^{\lambda_{\mathsf{max}}}\neq 
\varnothing$, respectively. 
  Now let 
$\overline{q}_{\mathsf{max}}\in Z^{\lambda_{\mathsf{max}}}$.
We denote 
$\overline{t}_{\mathsf{max}} = T(\overline{q}_{\mathsf{max}})$. By 
using \eqref{e:defPFqlam}, we show that
\begin{align*}
g(\overline{q}_{\mathsf{max}})  
&= \sum_s p_s \sum_{(a,b)\in \supp 
\overline{q}_{\mathsf{max}}^s}\overline{q}_{ab,\mathsf{max}}^s  \ln 
\frac{\overline{q}_{ab,\mathsf{max}}^s}{\overline{t}_{b,\mathsf{max}}}\\
&= \sum_s p_s \sum_{(a,b)\in \supp 
\overline{q}_{\mathsf{max}}^s}\overline{q}_{ab,\mathsf{max}}^s \ln 
\frac{\overline{t}_{b,\mathsf{max}} 
e^{-c_{ab}^s/{\lambda_{\mathsf{max}}} 
}}{\overline{t}_{b,\mathsf{max}}\sum_{a'b'}\overline{t}_{b',\mathsf{max}} e^{- c_{a'b'}^s/{\lambda_{\mathsf{max}}} } } \label{}\\
&=- \sum_s p_s \sum_{(a,b)\in \supp 
\overline{q}_{\mathsf{max}}^s}\overline{q}_{ab,\mathsf{max}}^s  
c_{ab}^s/{\lambda_{\mathsf{max}}}   -  \sum_s p_s \sum_{(a,b)\in 
\supp \overline{q}_{\mathsf{max}}^s}\overline{q}_{ab,\mathsf{max}}^s 
\ln  \sum_{a'b'}\overline{t}_{b',\mathsf{max}} e^{- 
c_{a'b'}^s/{\lambda_{\mathsf{max}}} } \\
&<- \sum_s p_s \sum_{(a,b)\in \supp 
\overline{q}_{\mathsf{max}}^s}\overline{q}_{ab,\mathsf{max}}^s  
c^s_{\mathsf{min}}/{\lambda_{\mathsf{max}}}   -  \sum_s p_s 
\sum_{(a,b)\in \supp 
\overline{q}_{\mathsf{max}}^s}\overline{q}_{ab,\mathsf{max}}^s \ln  
\sum_{a'b'}\overline{t}_{b',\mathsf{max}} e^{- 
c^s_{\mathsf{max}}/{\lambda_{\mathsf{max}}} } \\
&=- \frac{c_{\mathsf{min}}}{\lambda_{\mathsf{max}}}    -  \sum_s p_s 
\sum_{(a,b)\in \supp 
\overline{q}_{\mathsf{max}}^s}\overline{q}_{ab,\mathsf{max}}^s \ln  
e^{- c^s_{\mathsf{max}}/{\lambda_{\mathsf{max}}} }  - \ln |A|\\
&=  \frac{c_{\mathsf{max}} - 
c_{\mathsf{min}}}{\lambda_{\mathsf{max}}}   - \ln |A|\\
&=0.
\end{align*}
Finally, for every $\overline{q}_{0}\in 
Z^{0}$, the assumption that $c_{\mathsf{min}}$ is not attainable 
yields 
$g(\overline{q}_{0})>0$. \qed
\end{proof}

We describe our algorithm for solving Problem \ref{prob:main}.

\begin{algorithm}\label{Heuristic1}
Let ${\lambda}^+_0 = \lambda_{\mathsf{max}}$, ${\lambda}^-_0 = 
0$, and $q_{0}\in\reli (\Delta(A \times B))^S$. We consider 
$0<\varepsilon_f<\varepsilon_b$ and the 
sequences generated by the recurrence
\begin{align}
\label{e:algotseng}
&\begin{array}{l}
\text{for }k=0,1,2,\ldots\\
\left\lfloor
\begin{array}{l}
\text{While } \lambda^+_{k} - \lambda^-_{k} \geq \varepsilon_b\\
\lambda_{k} = \frac12 ({\lambda}^+_{k} + {\lambda}^-_{k})\\
\begin{array}{l}
\text{for }n=0,1,2,\ldots\\
\left\lfloor
\begin{array}{l}
\text{For }b\in B \\
\left\lfloor 
\begin{array}{ll}
{t_{b,n} =\sum\limits_{s\in S}\sum\limits_{a\in 
A}p_sq_{ab,n}^{s}}\\[3pt]
\end{array}
\right. \\
{\text{For }(a,b,s)\in A\times B\times S} \\
\left\lfloor 
\begin{array}{ll}
q_{ab,n+1}^s = \dfrac{t_{b,n} e^{-c_{ab}^s/\lambda_k 
}}{\sum\limits_{a'\in A}\sum\limits_{b'\in B}t_{b'\!,n} 
e^{- c_{a'b'}^s/\lambda_k } }\\[3pt]
\end{array}
\right.\\[5pt]
\text{STOP when } \|q_{n+1} - q_{n}\|_2 < \varepsilon_f
\end{array}
\right.\\
\end{array}\\
q_{k} = q_{n+1}\\
\text{If } \qquad \quad g(q_{k})<-0_{\mathsf{n}},\quad  \text{ then } {\lambda}^+_{k+1} = \lambda_k \text{ and }{\lambda}^-_{k+1} = {\lambda}^-_{k},\\
\text{Else if } \quad g(q_{k})>0_{\mathsf{n}},\quad \text{ then } {\lambda}^+_{k+1} = {\lambda}^+_{k} \text{ and }{\lambda}^-_{k+1} = \lambda_k,\\ 
\text{Else if } \quad g(q_{k})\in[-0_{\mathsf{n}},0_{\mathsf{n}}], \quad \text{ then } \lambda^-_{k+1}=\lambda^+_{k+1}= \lambda_k,\\
k = k+1,\\
\end{array}
\right.
\end{array}
\end{align}
where $0_{\mathsf{n}}>0$.
\end{algorithm}
In Algorithm~\ref{Heuristic1}, $0_{\mathsf{n}}$ depends on the 
precision of the computer and we set it as $10^{-20}$.

\begin{proposition}
\label{e:convlam}
Suppose that $c_{\mathsf{min}}$ is not attainable, consider the 
sequences $(\lambda_k)_{k\in\NN}$ and $(q_k)_{k\in\NN}$ 
generated by Algorithm~\ref{Heuristic1} {assuming exact 
solutions in the inner iterations} and set
\begin{equation}
\label{e:defk0}
k_0=\left[\log_2 \left(\frac{c_{\mathsf{max}} - 
c_{\mathsf{min}}}{\varepsilon_b \ln |A|}\right)\right]+1.
\end{equation}
 Then, for every $k\ge k_0$, 
$$|\lambda_k-\overline{\lambda}|<\varepsilon_b,$$
where $\overline{\lambda}>0$ is such that $\overline{q}\in 
Z^{\overline{\lambda}}$ is a solution to Problem~\ref{prob:main} 
satisfying $g(\overline{q})=0$, in view of
Proposition~\ref{p:saddle}\eqref{p:saddleii}.
\end{proposition}
\begin{proof}
The existence of $\overline{\lambda}$ follows from 
Proposition~\ref{p:saddle}\eqref{p:saddleii}. Moreover, we deduce 
from Lemma~\ref{lemma:LambdaMax} and 
Algorithm~\ref{Heuristic1} that 
\begin{equation}
\label{e:bound}
(\forall k\in\NN{\setminus\{0\}})\quad \lambda^+_{k} - 
\lambda^-_{k} \; {\leq} \;\frac{\lambda_{\mathsf{max}}}{2^k}
\end{equation}
and $g(Q(\lambda^+_{k}))<0$ and 
$g(Q(\lambda^-_{k}))>0$, where $Q(\lambda)$ denotes the unique 
limit of the subroutine in 
Algorithm~\ref{Heuristic1} when $\lambda_k=\lambda$.
Since Proposition~\ref{p:saddle}\eqref{p:saddlei} implies 
that 
$|\lambda_k-\overline{\lambda}|\le \lambda^+_{k} - 
\lambda^-_{k}$ and we have
$$(\forall k\in\NN)\quad 
\frac{\lambda_{\mathsf{max}}}{2^k}<\varepsilon_b\quad 
\Leftrightarrow\quad k> 
\log_2\left(\frac{\lambda_{\mathsf{max}}}{\varepsilon_b}\right),$$
 the result follows from \eqref{e:bound} and  
the definition of $\lambda_{\mathsf{max}}$ in 
Lemma~\ref{lemma:LambdaMax}.\qed
\end{proof}

\begin{remark}
Note that, when the cost satisfies {$c_{ab}^s = c_b^s \in 
\mathbb{R}$ for all $(a, b, s) \in A \times B \times S$,}  
Algorithm~\ref{Heuristic1} encompasses the Blahut-Arimoto 
procedure \cite{Arimoto72,Blahut72} combined with a 
bisection 
method. In this setting, Proposition~\ref{e:convlam} provides
{an estimation on the number of iterations to achieve a tolerance 
$\varepsilon_b$}.
\end{remark}

Proposition~\ref{e:convlam} asserts that we can approximate the 
unique
Lagrange multiplier $\overline{\lambda}$ with a tolerance 
$\varepsilon_b>0$ with $\lambda_{k_0}$ for $k_0$ defined in 
\eqref{e:defk0}. Then, $q_{k_0}$ is an approximation of a solution 
to 
Problem~\ref{prob:main}, in view of Proposition~\ref{p:saddle}.
{In the next section,} we provide some numerical experiments 
illustrating 
the advantages of the proposed approach.

\subsection{Numerical results}\label{sec:numerical}

In this section we investigate the performances of the proposed 
Algorithm~\ref{Heuristic1}. We consider as benchmark the 
fmincon-interior point function in MATLAB.

\subsubsection{Coordination game in 
{\cite[Example~2.1]{GossnerHernandezNeyman06}}}\label{sec:GHN}

We consider the cost function of the main example in 
{\cite[Example~2.1]{GossnerHernandezNeyman06}}, where 
$S=\{s_0,s_1\}$, 
$p_{s_0}=p_{s_1}=1/2$, and the cost function is given by the table 
below.
\begin{center}
\begin{tikzpicture}[xscale=1.1,yscale=0.9]
\draw [thick] (0,0) grid (2,2);
\draw (0.5,1.5) node{$0$};
\draw (1.5,1.5) node{$1$};
\draw (0.5,0.5) node{$1$};
\draw (1.5,0.5) node{$1$};
\draw (0,1.5) node[left]{$a_0$};
\draw (0,0.5) node[left]{$a_1$};
\draw (0.5,2) node[above]{$b_0$};
\draw (1.5,2) node[above]{$b_1$};
\draw (1,-0.2) node[below]{$s_0$};
\draw [thick] (4,0) grid (6,2);
\draw (4.5,1.5) node{$1$};
\draw (5.5,1.5) node{$1$};
\draw (4.5,0.5) node{$1$};
\draw (5.5,0.5) node{$0$};
\draw (4,1.5) node[left]{$a_0$};
\draw (4,0.5) node[left]{$a_1$};
\draw (4.5,2) node[above]{$b_0$};
\draw (5.5,2) node[above]{$b_1$};
\draw (5,-0.2) node[below]{$s_1$};
\end{tikzpicture}
\end{center}
Since $I^{s_0} = \{(a_0,b_0)\}$ and $I^{s_1} = \{(a_1,b_1)\}$ are 
singleton, for any distribution $q\in \Delta_0$, we have
$(T_0(q))_{b_0}=(T_0(q))_{b_1}=1/2$ and $g(q) =  
p_{s_0} 
\psi(q_{a_{0}b_{0}}^{s_0}, (T_0(q))_{b_0}) + p_{s_1} 
\psi(q_{a_{1}b_{1}}^{s_1}, (T_0(q))_{b_1}) = \frac12 \ln2 + \frac12 
\ln2 = \ln 2>0$, which implies that $c_{\textsf{min}}$ is not 
attainable.  We set $q_0\in {\rm ri}(\Delta(A\times B))^S$ as the 
uniform distribution, we consider {$\varepsilon_b=10^{-10}$} as 
the precision 
for the bisection algorithm in Algorithm~\ref{Heuristic1}
and for {the function fmincon}, and 
{$\varepsilon_f=10^{-12}$} as the precision for the fixed point 
procedure in \eqref{e:defPFqlam} and for the fmincon 
constraint.

Figure~\ref{fig:Fig_ErrorGHNa} shows that 
Algorithm~\ref{Heuristic1} and 
{the function fmincon} converge to the optimal 
solution, which 
is given by the distribution $\overline{q}$ defined by
{
\begin{center}
\begin{tikzpicture}[xscale=1.1,yscale=0.9]
\draw [thick] (0,0) grid (2,2);
\draw (0.5,1.5) node{$\gamma$};
\draw (1.5,1.5) node{$\frac{1-\gamma}{3}$};
\draw (0.5,0.5) node{$\frac{1-\gamma}{3}$};
\draw (1.5,0.5) node{$\frac{1-\gamma}{3}$};
\draw (0,1.5) node[left]{$a_0$};
\draw (0,0.5) node[left]{$a_1$};
\draw (0.5,2) node[above]{$b_0$};
\draw (1.5,2) node[above]{$b_1$};
\draw (1,-0.2) node[below]{$s_0$};
\draw [thick] (4,0) grid (6,2);
\draw (4.5,1.5) node{$\frac{1-\gamma}{3}$};
\draw (5.5,1.5) node{$\frac{1-\gamma}{3}$};
\draw (4.5,0.5) node{$\frac{1-\gamma}{3}$};
\draw (5.5,0.5) node{$\gamma$};
\draw (4,1.5) node[left]{$a_0$};
\draw (4,0.5) node[left]{$a_1$};
\draw (4.5,2) node[above]{$b_0$};
\draw (5.5,2) node[above]{$b_1$};
\draw (5,-0.2) node[below]{$s_1$};
\end{tikzpicture}
\end{center}
}
where $\gamma$ is the unique solution of 
{$-\gamma\ln\gamma-(1-\gamma)\ln(1-\gamma) + 
(1-\gamma) \ln 3 = \ln 2$, which is $\gamma \simeq 0.81071$, see \cite[Section~2.1]{GossnerHernandezNeyman06}.}

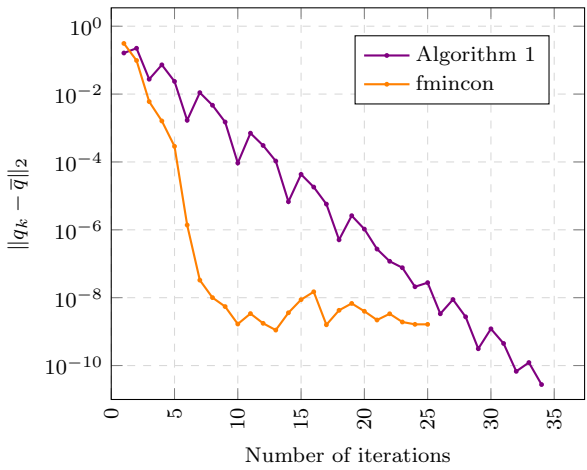
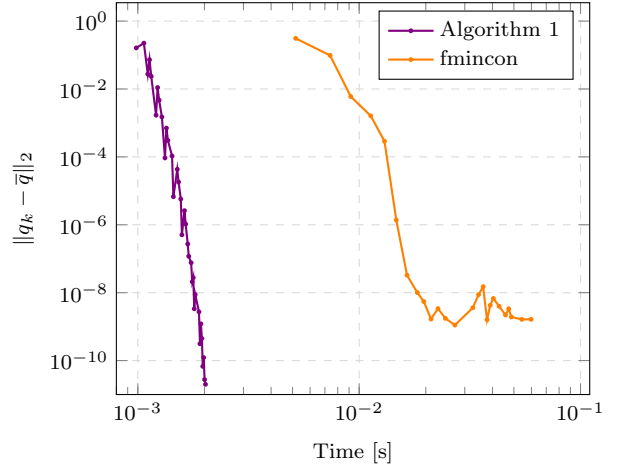
\begin{figure}[h!]
\centering
\begin{subfigure}{0.49\textwidth}
\begin{tikzpicture}
\begin{semilogyaxis}[
          width=1*\linewidth, 
          grid=major, 
          grid style={dashed,gray!30}, 
          xlabel=Number of iterations, 
          ylabel=$\|q_{k} - \overline{q}\|_2$,
		 xmin=2e-3,
		 ymin=1e-11,
          legend cell align=left,
          legend style={at={(0.72,0.93)},anchor=north}, 
          x tick label style={rotate=90,anchor=east} 
        ]
\addplot[thick, color=violet] plot[mark=*, mark size = 0.5pt] file 
{Fig_ErrorIterationGHN_Dichotomy5.txt};
\addplot[thick, color=orange] plot[mark=*, mark size = 0.5pt] file 
{Fig_ErrorIterationGHN_Fmincon5.txt};
\legend{Algorithm~\ref{Heuristic1},fmincon}
\end{semilogyaxis}
\end{tikzpicture}
\caption{Error term $(\|q_{k} - \overline{q}\|_2)_{k\in\NN}$ vs. number of iteration.}
\label{fig:Fig_ErrorGHNa}
\end{subfigure}
\hfill
\begin{subfigure}{0.49\textwidth}
\begin{tikzpicture}
\begin{loglogaxis}[
          width=1*\linewidth, 
          grid=major, 
          grid style={dashed,gray!30}, 
          xlabel=Time {[s]}, 
          ylabel=$\|q_{k} - \overline{q}\|_2$,
		 xmin=8e-4,
		 ymin=1e-11,
         xmax=1.1e-1,
          legend cell align=left,
          legend style={at={(0.76,0.98)},anchor=north}, 
          x tick label style={rotate=00,anchor=north} 
        ]
\addplot[thick, color=violet] plot[mark=*, mark size = 0.5pt] file 
{Fig_ErrorGHN_Dichotomy5.txt};
\addplot[thick, color=orange] plot[mark=*, mark size = 0.5pt] file 
{Fig_ErrorTimeGHN_Fmincon5.txt};
\legend{Algorithm~\ref{Heuristic1},fmincon}
\end{loglogaxis}
\end{tikzpicture}
\caption{Error term $(\|q_{k} - \overline{q}\|_2)_{k\in\NN}$ vs. computing time.}
\label{fig:Fig_ErrorGHNb}
\end{subfigure}
\caption{Comparison of Algorithm~\ref{Heuristic1} and {the 
function fmincon} in MATLAB.}\label{fig:Comparison}
\end{figure}

\begin{table}[!h] 
\begin{center}
\caption{Results obtained for {the function fmincon} and for 
Algorithm~\ref{Heuristic1}.}
\label{Table1}
\begin{tabular}{|c|c|c|c|c|c|}
\hline
&Value&$\lambda$&Iterations&Time [s]&Precision $\|q_k-\overline{q}\|_2$ \\
\hline
Algorithm~\ref{Heuristic1} 	& $0.18929$		&  	$0.39166$	&$34$ &	 
$0.00202$ 	& 	$1.9987^{-11}$\\
\hline
fmincon 		&	$0.18929$	&		$0.39166$	&	$25$ & $0.05857$		&  
$1.6389^{-10} $\\
\hline
\end{tabular}
\end{center}
\end{table}

The outputs of the algorithms are presented in Table~\ref{Table1} 
below. Algorithm~\ref{Heuristic1} is more efficient than {the 
function fmincon} in MATLAB.

\subsubsection{Extended coordination game}

We consider an extended version of the cost function of the 
example studied in the previous 
Section~\ref{sec:GHN}, with a larger dimension $d = |A| = |B| = |S| 
>2$. We set $p_s\equiv 1/d$ and the costs in 
Fig.~\ref{fig:ExtendedGNH}. In this case, for every
$s\in S$, the set $I^{s}$ is a singleton, {thus the distribution $q\in 
\Delta_0$ induces $T_0(q) $ which is 
uniform. Therefore $g(q)  = \ln d>0$ and thus $c_{\textsf{min}}$ is not attainable.} We compare the performances of Algorithm~\ref{Heuristic1} with 
{the function fmincon} in MATLAB.

\begin{figure}[!ht]
\begin{center}
\begin{tikzpicture}[xscale=0.7,yscale=0.6]
\draw [thick] (0,0) rectangle (5,5);
\draw [thick] (0,4) -- (5,4);
\draw [thick] (0,3) -- (5,3);
\draw [thick] (0,1) -- (5,1);
\draw [thick] (1,0) -- (1,5);
\draw [thick] (2,0) -- (2,5);
\draw [thick] (4,0) -- (4,5);
\draw (0.5,4.5) node{$0$};
\draw (1.5,4.5) node{$1$};
\draw (0.5,3.5) node{$1$};
\draw (1.5,3.5) node{$1$};
\draw (1.5,0.5) node{$1$};
\draw (0.5,0.5) node{$1$};
\draw (4.5,0.5) node{$1$};
\draw (4.5,3.5) node{$1$};
\draw (4.5,4.5) node{$1$};
\draw (0.5,2) node{$\vdots$};
\draw (1.5,2) node{$\vdots$};
\draw (4.5,2) node{$\vdots$};
\draw (3,4.5) node{$\ldots$};
\draw (3,3.5) node{$\ldots$};
\draw (3,0.5) node{$\ldots$};
\draw (3,2) node{$\ddots$};
\draw (0,4.5) node[left]{$a_1$};
\draw (0,3.5) node[left]{$a_2$};
\draw (0,0.5) node[left]{$a_d$};
\draw (0.5,5) node[above]{$b_1$};
\draw (1.5,5) node[above]{$b_2$};
\draw (4.5,5) node[above]{$b_d$};
\draw (2.5,-0.2) node[below]{$s_1$};
\draw [thick] (7,0) rectangle (12,5);
\draw [thick] (7,4) -- (12,4);
\draw [thick] (7,3) -- (12,3);
\draw [thick] (7,1) -- (12,1);
\draw [thick] (8,0) -- (8,5);
\draw [thick] (9,0) -- (9,5);
\draw [thick] (11,0) -- (11,5);
\draw (7.5,4.5) node{$1$};
\draw (8.5,4.5) node{$1$};
\draw (7.5,3.5) node{$1$};
\draw (8.5,3.5) node{$0$};
\draw (8.5,0.5) node{$1$};
\draw (7.5,0.5) node{$1$};
\draw (11.5,0.5) node{$1$};
\draw (11.5,3.5) node{$1$};
\draw (11.5,4.5) node{$1$};
\draw (7.5,2) node{$\vdots$};
\draw (8.5,2) node{$\vdots$};
\draw (11.5,2) node{$\vdots$};
\draw (10,4.5) node{$\ldots$};
\draw (10,3.5) node{$\ldots$};
\draw (10,0.5) node{$\ldots$};
\draw (10,2) node{$\ddots$};
\draw (7,4.5) node[left]{$a_1$};
\draw (7,3.5) node[left]{$a_2$};
\draw (7,0.5) node[left]{$a_d$};
\draw (7.5,5) node[above]{$b_1$};
\draw (8.5,5) node[above]{$b_2$};
\draw (11.5,5) node[above]{$b_d$};
\draw (9.5,-0.2) node[below]{$s_2$};
\draw (13,2.2) node{$\ldots$};
\draw [thick] (14,0) rectangle (19,5);
\draw [thick] (14,4) -- (19,4);
\draw [thick] (14,3) -- (19,3);
\draw [thick] (14,1) -- (19,1);
\draw [thick] (15,0) -- (15,5);
\draw [thick] (16,0) -- (16,5);
\draw [thick] (18,0) -- (18,5);
\draw (14.5,4.5) node{$1$};
\draw (15.5,4.5) node{$1$};
\draw (14.5,3.5) node{$1$};
\draw (15.5,3.5) node{$1$};
\draw (15.5,0.5) node{$1$};
\draw (14.5,0.5) node{$1$};
\draw (18.5,0.5) node{$0$};
\draw (18.5,3.5) node{$1$};
\draw (18.5,4.5) node{$1$};
\draw (14.5,2) node{$\vdots$};
\draw (15.5,2) node{$\vdots$};
\draw (18.5,2) node{$\vdots$};
\draw (17,4.5) node{$\ldots$};
\draw (17,3.5) node{$\ldots$};
\draw (17,0.5) node{$\ldots$};
\draw (17,2) node{$\ddots$};
\draw (14,4.5) node[left]{$a_1$};
\draw (14,3.5) node[left]{$a_2$};
\draw (14,0.5) node[left]{$a_d$};
\draw (14.5,5) node[above]{$b_1$};
\draw (15.5,5) node[above]{$b_2$};
\draw (18.5,5) node[above]{$b_d$};
\draw (16.5,-0.2) node[below]{$s_d$};
\end{tikzpicture}
\end{center}
\caption{Extended version of the cost function of \cite{GossnerHernandezNeyman06}.}\label{fig:ExtendedGNH}
\end{figure}

\begin{figure}[h!]
\begin{center}
\begin{tikzpicture}
\begin{semilogyaxis}[
          width=0.5*\linewidth, 
          grid=major, 
          grid style={dashed,gray!30}, 
          xlabel=Number of dimensions $d\in{[2,18]}$, 
          ylabel=Times {[s]},
          legend cell align=left,
          legend style={at={(0.25,0.97)},anchor=north}, 
          x tick label style={rotate=90,anchor=east} 
        ]
\addplot[thick, color=violet] plot[mark=*, mark size = 0.5pt] file 
{Fig_ExtendedGHN_Dichotomy_OK.txt};
\addplot[thick, color=orange] plot[mark=+, mark size = 1pt] file 
{Fig_ExtendedGHN_Fmincon_OK.txt};
\legend{Algorithm~\ref{Heuristic1},fmincon}
\end{semilogyaxis}
\end{tikzpicture}
\caption{Comparison of the performances of 
Algorithm~\ref{Heuristic1} and of {the function 
fmincon}.}\label{fig:ExtendedGHN}
\end{center}
\end{figure}
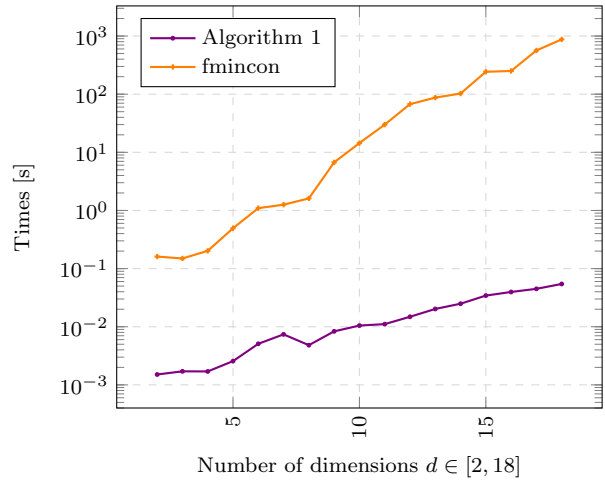

\begin{table}[!h] 
\begin{center}
\caption{Efficiency of Algorithm~\ref{Heuristic1} compared to {the 
function fmincon}.}
\label{Table4}
\begin{filecontents*}{Table_ExtendedGHN2_OK.txt}
2	0.001509	0.161	256
3	0.0017103	0.14942	427
4	0.0017034	0.2019	1434
5	0.0025616	0.494	4705
6	0.00509	1.0954	9162
7	0.0073912	1.2567	10693
8	0.0048182	1.6119	11294
9	0.0083331	6.7145	35106
10	0.010473	14.402	53107
11	0.01109	29.954	70641
12	0.014838	67.423	1.2111e+05
13	0.020233	87.304	1.2094e+05
14	0.024894	102.52	1.1258e+05
15	0.034343	242.9	1.9924e+05
16	0.039618	250.66	1.639e+05
17	0.044878	565.24	3.0963e+05
18	0.054421	874.71	3.8503e+05
\end{filecontents*}
\pgfplotstabletypeset[
column type=c,
every head row/.style={
before row={
\hline
& \multicolumn{1}{c|}{Algorithm~\ref{Heuristic1}} & \multicolumn{2}{c|}{fmincon}\\},
},
    columns/0/.style={column name={Dimension: $d$}},
    columns/1/.style={column name={Time {[s]} }},
    columns/2/.style={column name={Time {[s]}}},
    columns/3/.style={column name={Iterations}},
    header=false,
    string type,
    before row=\hline,
     every last row/.style={after row=\hline},
    column type/.add={|}{},
    every last column/.style={column type/.add={}{|}}]
    {Table_ExtendedGHN2_OK.txt}%
\end{center}
\end{table}

For Algorithm~\ref{Heuristic1}, we set the uniform distribution 
$q_0\equiv (|A||B|)^{-1}\in\reli(\Delta(A\times B))^S$. Because of 
the symmetry of the 
cost 
function in Figure~\ref{fig:ExtendedGNH}, only $2$ iterations are 
required in the fixed-point recurrence \eqref{e:defPFqlam}. 
Note that Proposition~\ref{e:convlam} asserts that the number of 
iterations of 
Algorithm~\ref{Heuristic1} 
decreases with $|A|$. Indeed, by considering 
$\varepsilon_b=10^{-10}$
and $\varepsilon_f=10^{-11}$, \eqref{e:defk0} leads to
$34$ 
iterations when $|A|\in\{2,3\}$, $33$ iterations when 
$|A|\in\{4,\ldots,10\}$, and $32$ iterations when 
$|A|\in\{11,\ldots,18\}$. These numbers are very close to the 
number of iterations to achieve the tolerance $\varepsilon_b$, 
numerically, even though the error $\varepsilon_f$ is not 
considered in Proposition~\ref{e:convlam}. 

The outputs of the algorithms are presented in 
Fig.~\ref{fig:ExtendedGHN} and in Table~\ref{Table4}. For all the 
considered dimensions $d\in\NN$, our proposed 
Algorithm~\ref{Heuristic1} is far more efficient than {the function 
fmincon} in MATLAB.


\subsubsection{Additional examples in higher
dimensions}\label{sec:dimension10}

In this section, we increase the dimension of the cost functions. 
For each of the values for $(A,B,S)$ in Table~\ref{Table10_a}, we 
draw at random $10$ cost functions where the cost values are 
taken uniformly in $[0,10]$ with one decimal 
value. We consider the precision $\varepsilon_b=10^{-6}$ for the 
bisection method and $\varepsilon_f=10^{-8}$ for the fixed point 
procedure in \eqref{e:defPFqlam}. 
We randomly generate $10$ cost functions such that 
$c_{\mathsf{min}}$ is not 
attainable. In Table~\ref{Table10_a}, we provide the average 
computational time of Algorithm~\ref{Heuristic1} to achieve the 
tolerance $\varepsilon_b=10^{-6}$ together with the average 
computational cost of the subroutine \eqref{e:defPFqlam}.
When the dimension of the 
problem is  $32 \,000$, 
Table~\ref{Table10_a} and 
Fig.~\ref{fig:Dichotomy_20_40_40_Exp-6_a} show that 
Algorithm~\ref{Heuristic1} converges approximately within $6$ 
minutes on average.

\begin{table}[!h] 
\begin{center}
\caption{Average convergence time of  Algorithm~\ref{Heuristic1} for $10$ randomly selected costs functions.}
\label{Table10_a}
\begin{filecontents*}{Table_AverageDichotomy_20_40_40_Exp-6.txt}
{(5, 10, 10)}	500	1.2328	0.0535	987.64
{(5, 10, 20)}	1000	2.6742	0.1162	1066.7
{(5, 20, 10)}	1000	3.2391	0.1408	1314.57
{(5, 20, 20)}	2000	13.4459	0.5845	3078.63
{(10, 20, 20)}	4000	16.4821	0.7164	1775.48
{(10, 20, 40)}	8000	29.173	1.2681	1550.86
{(10, 40, 20)}	8000	39.2039	1.7043	2039.83
{(10, 40, 40)}	16000	125.0127	5.4347	2851.39
{(20, 40, 40)}	32000	370.3015	16.0986	3611.62
\end{filecontents*}
\pgfplotstabletypeset[
column type=c,
every head row/.style={
before row={
\hline
&& \multicolumn{1}{c|}{Algorithm~\ref{Heuristic1}} & \multicolumn{2}{c|}{Fixed Point \eqref{e:defPFqlam}}\\},
},
    columns/0/.style={column name={($A,B,S$)}},
    columns/1/.style={column name={Dimension}},
    columns/2/.style={column name={Av. time {[s]}}},
    columns/3/.style={column name={Av. time {[s]}}},
    columns/4/.style={column name={Av. iterations}},
    header=false,
    string type,
    before row=\hline,
     every last row/.style={after row=\hline},
    column type/.add={|}{},
    every last column/.style={column type/.add={}{|}}]
    {Table_AverageDichotomy_20_40_40_Exp-6.txt}%
\end{center}
\end{table}
We observe that the computational time increases drastically when 
the dimension of $S$ is increased. Moreover, for a same 
{problem dimension}, the computational time is more 
expensive when the dimension of $B$ is larger. For every 
dimension, the computational time and number of iterations for the 
fixed point procedure is tractable.


\begin{figure}[h!]
\centering
\begin{tikzpicture}
\begin{loglogaxis}[
          width=0.83*\linewidth, 
          height=7cm,
          grid=major, 
          grid style={dashed,gray!30}, 
          xlabel=Time {[s]}, 
          ylabel=$\|q_{k+1} - q_{k}\|_2$,
		 xmin=10e-3,
         legend cell align=left,
         legend style={at={(1.07,0.65)},anchor=south}, 
          x tick label style={rotate=90,anchor=east} 
        ]
\addplot[thick, color=cyan] plot[mark=*, mark size = 0.8pt] file 
{Fig_Dichotomy_5_10_10_Exp-6_ErrorQ_1.txt};
\addplot[thick, color=orange] plot[mark=*,  mark size = 0.8pt] file 
{Fig_Dichotomy_10_20_20_Exp-6_ErrorQ_1.txt};
\addplot[thick, color=violet] plot[mark=*,  mark size = 0.8pt] file 
{Fig_Dichotomy_20_40_40_Exp-6_ErrorQ_1.txt};
\addplot[thick, color=cyan] plot[mark=*, mark size = 0.8pt] file 
{Fig_Dichotomy_5_10_10_Exp-6_ErrorQ_2.txt};
\addplot[thick, color=orange] plot[mark=*,  mark size = 0.8pt] file 
{Fig_Dichotomy_10_20_20_Exp-6_ErrorQ_2.txt};
\addplot[thick, color=violet] plot[mark=*,  mark size = 0.8pt] file 
{Fig_Dichotomy_20_40_40_Exp-6_ErrorQ_2.txt};
\addplot[thick, color=cyan] plot[mark=*, mark size = 0.8pt] file 
{Fig_Dichotomy_5_10_10_Exp-6_ErrorQ_3.txt};
\addplot[thick, color=orange] plot[mark=*,  mark size = 0.8pt] file 
{Fig_Dichotomy_10_20_20_Exp-6_ErrorQ_3.txt};
\addplot[thick, color=violet] plot[mark=*,  mark size = 0.8pt] file 
{Fig_Dichotomy_20_40_40_Exp-6_ErrorQ_3.txt};
%
\addplot[thick, color=cyan] plot[mark=*, mark size = 0.8pt] file 
{Fig_Dichotomy_5_10_10_Exp-6_ErrorQ_4.txt};
\addplot[thick, color=orange] plot[mark=*,  mark size = 0.8pt] file 
{Fig_Dichotomy_10_20_20_Exp-6_ErrorQ_4.txt};
\addplot[thick, color=violet] plot[mark=*,  mark size = 0.8pt] file 
{Fig_Dichotomy_20_40_40_Exp-6_ErrorQ_4.txt};
\addplot[thick, color=cyan] plot[mark=*, mark size = 0.8pt] file 
{Fig_Dichotomy_5_10_10_Exp-6_ErrorQ_5.txt};
\addplot[thick, color=orange] plot[mark=*,  mark size = 0.8pt] file 
{Fig_Dichotomy_10_20_20_Exp-6_ErrorQ_5.txt};
\addplot[thick, color=violet] plot[mark=*,  mark size = 0.8pt] file 
{Fig_Dichotomy_20_40_40_Exp-6_ErrorQ_5.txt};
\addplot[thick, color=cyan] plot[mark=*, mark size = 0.8pt] file 
{Fig_Dichotomy_5_10_10_Exp-6_ErrorQ_6.txt};
\addplot[thick, color=orange] plot[mark=*,  mark size = 0.8pt] file 
{Fig_Dichotomy_10_20_20_Exp-6_ErrorQ_6.txt};
\addplot[thick, color=violet] plot[mark=*,  mark size = 0.8pt] file 
{Fig_Dichotomy_20_40_40_Exp-6_ErrorQ_6.txt};
%
\addplot[thick, color=cyan] plot[mark=*,  mark size = 0.8pt]file 
{Fig_Dichotomy_5_10_10_Exp-6_ErrorQ_7.txt};
\addplot[thick, color=orange] plot[mark=*,  mark size = 0.8pt] file 
{Fig_Dichotomy_10_20_20_Exp-6_ErrorQ_7.txt};
\addplot[thick, color=violet] plot[mark=*,  mark size = 0.8pt] file 
{Fig_Dichotomy_20_40_40_Exp-6_ErrorQ_7.txt};
\addplot[thick, color=cyan] plot[mark=*, mark size = 0.8pt] file 
{Fig_Dichotomy_5_10_10_Exp-6_ErrorQ_8.txt};
\addplot[thick, color=orange] plot[mark=*,  mark size = 0.8pt] file 
{Fig_Dichotomy_10_20_20_Exp-6_ErrorQ_8.txt};
\addplot[thick, color=violet] plot[mark=*,  mark size = 0.8pt] file 
{Fig_Dichotomy_20_40_40_Exp-6_ErrorQ_8.txt};
\addplot[thick, color=cyan] plot[mark=*, mark size = 0.8pt] file 
{Fig_Dichotomy_5_10_10_Exp-6_ErrorQ_9.txt};
\addplot[thick, color=orange] plot[mark=*,  mark size = 0.8pt] file 
{Fig_Dichotomy_10_20_20_Exp-6_ErrorQ_9.txt};
\addplot[thick, color=violet] plot[mark=*,  mark size = 0.8pt] file 
{Fig_Dichotomy_20_40_40_Exp-6_ErrorQ_9.txt};
\addplot[thick, color=cyan] plot[mark=*, mark size = 0.8pt] file 
{Fig_Dichotomy_5_10_10_Exp-6_ErrorQ_10.txt};
\addplot[thick, color=orange] plot[mark=*,  mark size = 0.8pt] file 
{Fig_Dichotomy_10_20_20_Exp-6_ErrorQ_10.txt};
\addplot[thick, color=violet] plot[mark=*,  mark size = 0.8pt] file 
{Fig_Dichotomy_20_40_40_Exp-6_ErrorQ_10.txt};
%
%
\legend{{$(A,B,S)=(5,10,10)$},{$(A,B,S)=(10,20,20)$},{$(A,B,S)=(20,40,40)$}}
\end{loglogaxis}
\end{tikzpicture}
\caption{Convergence time of  Algorithm~\ref{Heuristic1} for $10$ 
randomly selected {cost} functions.}
\label{fig:Dichotomy_20_40_40_Exp-6_a}
\end{figure}
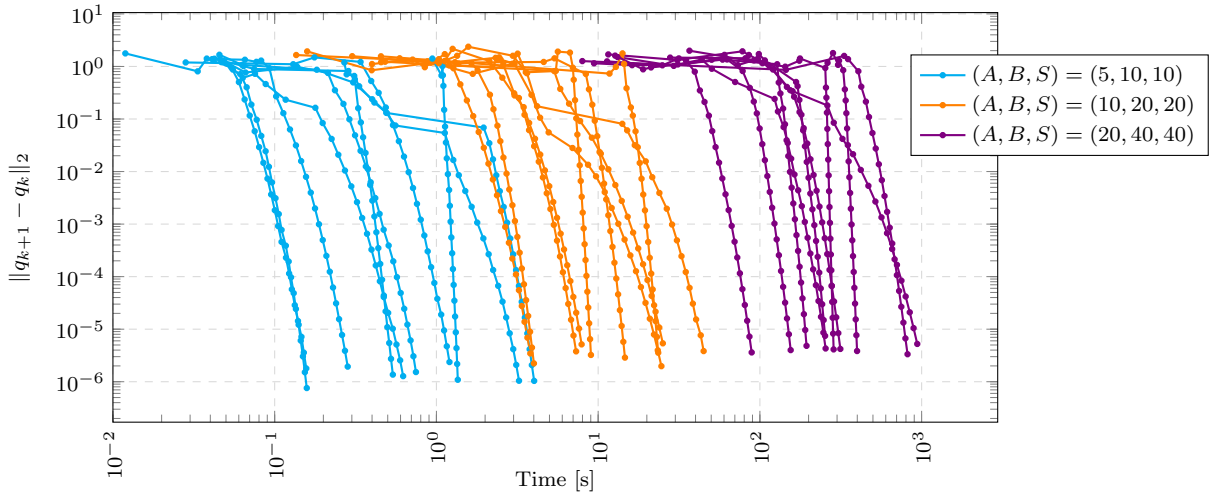

\section{Conclusions}
We proposed a Bregman proximal gradient method for linear 
optimization problems with an entropic constraint, covering both 
active and inactive constraint cases. When the constraint is 
inactive, a fixed-point iteration solves the problem efficiently. In the 
active case, we combined the method with a bisection scheme to 
estimate the Lagrange multiplier, recovering the Blahut-Arimoto 
algorithm as a special case.

The method achieves an $O(1/n)$ {convergence} rate and 
outperforms standard solvers in numerical experiments, including 
high-dimensional settings. This highlights its relevance for 
information-theoretic and game-theoretic applications.


{\small
\textbf{Acknowledgments.}
This work was supported by the National 
Agence of Research and Development (ANID) from Chile, under 
the grants FONDECYT 1230257, MATH-AmSud 23-MATH-17, 
and Centro de Modelamiento Matem\'{a}tico (CMM) BASAL fund 
FB210005 for centers of excellence.
This work was funded by the French government under the France 2030 ANR program ``PEPR Networks of the Future'' (ref. ANR-22-PEFT-0010). The authors gratefully acknowledge financial support from the iVisit program of IRISA UMR 6074.}





\end{document}